\documentclass[11pt,a4paper,reqno]{amsart}
\usepackage{fullpage}
\usepackage{amsmath,amssymb,amsthm,mathtools,mathrsfs}
\usepackage[colorlinks=true,linkcolor=blue,citecolor=blue,urlcolor=blue]{hyperref}
\usepackage[nameinlink,noabbrev]{cleveref}

\hypersetup{
pdftitle={Scattering Across the Long-Range Threshold for One-Dimensional Nonlinear Schr\"odinger Equations},
pdfauthor={Yonggeun Cho and Jinyeop Lee},
pdfsubject={Uniform scattering asymptotics across the long-range threshold for one-dimensional nonlinear Schr\"odinger equations},
pdfkeywords={nonlinear Schr\"odinger equation, modified scattering, long-range scattering, scattering transition, asymptotic profile}
}

\numberwithin{equation}{section}
\crefname{equation}{}{}
\newtheorem{theorem}{Theorem}[section]
\newtheorem{proposition}[theorem]{Proposition}
\newtheorem{lemma}[theorem]{Lemma}
\newtheorem{corollary}[theorem]{Corollary}

\theoremstyle{remark}
\newtheorem{remark}[theorem]{Remark}
\crefname{theorem}{Theorem}{Theorems}
\crefname{proposition}{Proposition}{Propositions}
\crefname{lemma}{Lemma}{Lemmas}
\crefname{corollary}{Corollary}{Corollaries}
\crefname{definition}{Definition}{Definitions}
\crefname{remark}{Remark}{Remarks}

\newcommand{\R}{\mathbb R}
\newcommand{\C}{\mathbb C}
\newcommand{\ii}{\mathrm i}
\newcommand{\ee}{\mathrm e}
\newcommand{\dd}{\,\mathrm d}
\newcommand{\norm}[1]{\left\lVert #1\right\rVert}
\newcommand{\snorm}[1]{\lVert #1\rVert}
\newcommand{\abs}[1]{\left\lvert #1\right\rvert}

\title{Scattering Across the Long-Range Threshold\\
for One-Dimensional Nonlinear Schr\"odinger Equations}
\author{Yonggeun Cho}
\address[Y. Cho]{Department of Mathematics and Institute of Pure and Applied Mathematics, Jeonbuk National University, Jeonju 54896, Republic of Korea}
\email{\href{mailto:changocho@jbnu.ac.kr}{changocho@jbnu.ac.kr}}
\author{Jinyeop Lee}
\address[J. Lee]{Department of Applied Mathematics, Kyung Hee University, 1732 Deogyeong-daero, Giheung-gu, Yongin-si, Gyeonggi-do, South Korea}
\email{\href{mailto:jinyeop.lee@khu.ac.kr}{jinyeop.lee@khu.ac.kr}}
\date{}

\begin{document}

\begin{abstract}
	We study small solutions to the one-dimensional nonlinear Schr\"odinger family
	\[
	\ii\partial_tu_\delta=-\partial_x^2u_\delta+\kappa\abs{u_\delta}^{2+\delta}u_\delta,
	\qquad 0\leq\delta\leq\delta_0,
	\]
	uniformly as $\delta\to0^+$ and $t\to\infty$, across the threshold between cubic modified scattering and nearby-power ordinary scattering. For initial data small in $H^{0,1}(\R)$, with one spatial weight and no physical derivative assumed, we prove a joint-limit asymptotic formula governed by an explicit nonlinear clock. The variable $ \delta\log t$ gives three regimes according as it tends to zero, a positive finite limit, or infinity. Within the regime $\delta\log t\to0$, successive Taylor terms of the clock become visible at an infinite hierarchy of time scales,
	each requiring finer absolute phase accuracy. These scales accumulate at the transition scale, where the exact clock describes them together. In the transition and saturated regimes, the clock is of order $\delta^{-1}$, so the first-order variation of its coefficient contributes a finite phase correction.
\end{abstract}

\maketitle

\section{Introduction and history}
\label{sec:introduction-history}

We consider the Cauchy problem
\begin{equation}\label{eq:nls-delta}
	\begin{cases}
		\displaystyle
		\ii\partial_tu_\delta=-\partial_x^2u_\delta+\kappa\abs{u_\delta}^{2+\delta}u_\delta,\\[0.4em]
		u_\delta(0)=u_{\mathrm{in}},
	\end{cases}
\end{equation}
where $\kappa\in\{+1,-1\}$, $0\leq\delta\leq\delta_0$, and the initial datum is independent of $\delta$.

The fixed-parameter cases of \eqref{eq:nls-delta} belong to well-developed but qualitatively different scattering theories.  At $\delta=0$, the one-dimensional cubic equation is long range: the leading interaction accumulates logarithmically, and small localized solutions exhibit modified rather than ordinary scattering.  Modified wave operators, sharp decay, asymptotic completeness, and higher-order asymptotics were developed in \cite{Ozawa1991,HayashiNaumkin1998AJM,HayashiKaikinaNaumkin1998,KitaWada2002,LindbladSoffer2006,KatoPusateri2011}. Ifrim--Tataru proved global bounds and modified asymptotics for small data in $H^{0,1}$, with no physical derivative assumed, and also treated short-range modifications~\cite{IfrimTataru2015}.  Cubic modified scattering at this weighted regularity is therefore not the novelty here.

For every fixed $\delta>0$, the power is instead short range: linear decay $\abs{u(t)}\sim t^{-1/2}$ makes the leading nonlinear interaction of size $t^{-1-\delta/2}$, which is integrable.  The corresponding ordinary scattering and wave-operator theory is classical. See \cite{GinibreVelo1985,Cazenave2003,MoriyamaTonegawaTsutsumi2003}.  The problem of this paper is the singular joint limit
\[
	\delta\to0^+,
	\qquad
	t\to\infty.
\]
Fixed-parameter results alone do not give a common smallness threshold down to $\delta=0$, profiles constructed uniformly in $\delta$, or the first-order expansion in the exponent needed when the nonlinear phase has size $\delta^{-1}$.
Our recent work on the Yukawa--Coulomb Hartree equation provides a related motivation~\cite{ChoLee2026}: there the transition variable is $\mu t$, whereas here it is $\delta\log t$.  The singular prefactor $1/\delta$ in the present clock creates an additional hierarchy with no counterpart in that comparison.

The clock is already determined from the decay calculation above. The associated nonlinear clock is
\begin{equation}\label{eq:phase-clock-preview}
	P_\delta(t)
	:=
	\begin{cases}
		\displaystyle \frac{2}{\delta}\bigl(1-t^{-\delta/2}\bigr),&\delta>0,\\[0.8em]
		\log t,&\delta=0.
	\end{cases}
\end{equation}
For $\delta>0$, write
\[
	\lambda:=\delta\log t,
	\qquad
	P_\delta(t)=\frac{2(1-\ee^{-\lambda/2})}{\delta}.
\]
The transition scale is reached when $\lambda$ is of order one, equivalently
at physical times $t=\exp(O(1)/\delta)$.

The singular factor $1/\delta$ also reveals finer structure before that transition. Although $\delta\log t\to0$ gives $P_\delta(t)/\log t\to1$, the difference $P_\delta(t)-\log t$ need not tend to zero. A small relative clock error can therefore remain visible in the oscillatory phase. Expanding in $\delta$ gives
\[
	P_\delta(t)
	=
	\sum_{k=0}^{\infty}
	\frac{(-1)^k}{2^k(k+1)!}\,
	\delta^k(\log t)^{k+1}.
\]
The $m$-th correction becomes visible in absolute phase when $\delta^m(\log t)^{m+1}$ is of order one, hence when $\log t$ is of order $\delta^{-m/(m+1)}$.  In physical-time scale notation, this gives
\begin{equation}\label{eq:cubic-hierarchy-scales-preview}
	1
	\ll
	\ee^{\delta^{-1/2}}
	\ll
	\ee^{\delta^{-2/3}}
	\ll
	\ee^{\delta^{-3/4}}
	\ll
	\cdots
	\ll
	\ee^{\delta^{-m/(m+1)}}
	\ll
	\cdots
	\ll
	\ee^{O(1)/\delta}.
\end{equation}
Here \eqref{eq:cubic-hierarchy-scales-preview} is informal scale notation. The rigorous sequence conditions appear in \Cref{thm:main-time-scale-hierarchy}. The finite-$m$ clocks are cumulative Taylor truncations, not mutually exclusive scattering regimes.  They all lie inside $\delta\log t\to0$ and accumulate as $m/(m+1)\uparrow1$, where the exact clock includes all Taylor terms.
The three regimes remain
\[
\delta\log t\to0,
\qquad
\delta\log t\to\Lambda\in(0,\infty),
\qquad
\delta\log t\to\infty.
\]

The main results, \Cref{thm:main-time-scale-hierarchy,thm:exact_clock}, give physical-space asymptotics for this joint limit with small $H^{0,1}$ initial data. The first theorem describes the Taylor hierarchy and the scales at which successive terms contribute a finite phase. The second retains the exact clock and applies across all three regimes. Here $W_0$ is the cubic asymptotic profile, $Z_0$ is the limiting profile obtained after removing the evolving nonlinear phase, and $\abs{W_0}=\abs{Z_0}$. Throughout the paper, a dot denotes the first right derivative with respect to the exponent at $\delta=0$, in the topology specified in each statement. Along any joint sequence $\delta_n\to0^+$, $t_n\to\infty$, with $\lambda_n=\delta_n\log t_n$, \Cref{thm:exact_clock} gives
\[
\begin{aligned}
	u_{\delta_n}(t_n,x)
	=&(2t_n)^{-1/2}\ee^{\frac{\ii x^2}{4t_n}}W_0(x/2t_n)\\
	&\times\exp\Biggl\{-\ii\kappa\Biggl[
	\frac12P_{\delta_n}(t_n)\abs{W_0(x/2t_n)}^2\\
	&\hspace{5em}+(1-\ee^{-\lambda_n/2})\Bigl[
	2\operatorname{Re}\bigl(\overline{Z_0(x/2t_n)}\dot Z_0(x/2t_n)\bigr)\\
	&\hspace{7em}+\abs{Z_0(x/2t_n)}^2
	\left(\log\abs{Z_0(x/2t_n)}-\frac12\log2\right)
	\Bigr]\Biggr]\Biggr\}
	+o_{L_x^2}(1).
\end{aligned}
\]
The coefficient of the clock has the first-order expansion
\[
2^{-1-\delta/2}\abs{W_\delta}^{2+\delta}
=\frac12\abs{W_0}^2
+\frac{\delta}{2}\left[
2\operatorname{Re}(\overline{Z_0}\dot Z_0)
+\abs{Z_0}^2\left(\log\abs{Z_0}-\frac12\log2\right)
\right]+o_{L_v^2}(\delta).
\]
Logarithmic products are defined by continuous extension at zero. Since the clock is of order $\delta^{-1}$ in the transition and saturated regimes, this first-order variation contributes the finite phase correction shown explicitly above. The Taylor hierarchy instead comes from expanding the clock multiplying $\abs{W_0}^2/2$. The leading clock is retained in \Cref{thm:exact_clock}; replacing it by $2(1-\ee^{-\Lambda/2})/\delta$ or $2/\delta$ requires the stronger assumptions in \Cref{cor:simplified-clock-asymptotics}.

\subsection{Analytic ingredients and organization}
\label{subsec:introduction-strategy}

The analytic problem is to make this picture uniform down to $\delta=0$ from initial data in $H^{0,1}$, with no physical derivative assumed.  Endpoint Strichartz estimates and mass conservation give the mass-class construction, while the Galilean field propagates the spatial weight. In Fourier variables, this weight controls one derivative of the interaction profile. Conjugation by the small Schr\"odinger propagator separates the leading nonlinear phase from an integrable remainder. A coupled bootstrap then gives uniform decay and a limiting profile. Parameter quotients identify the first-order expansion of the asymptotic profile and the coefficient of the clock. Finally, the clock identities and Taylor remainder bounds turn these analytic estimates into the joint-limit formulas.

Sections~\ref{sec:uniform-preliminary-estimates}--\ref{sec:rescaled-amplitude-exact-phase} establish the uniform profile asymptotics, \Cref{sec:quantitative-profile-continuity} proves the first-order expansion in the exponent, and \Cref{sec:transition-phase-analysis} analyzes the exact clock, Taylor hierarchy, three regimes, and refinements.  \Cref{sec:proof-main-theorem} then assembles the two physical-space theorems.

\section{Main results}
\label{sec:main-results}

\subsection{Notation and function spaces}
\label{subsec:notation-function-spaces}

We use the weighted mass space
\begin{equation}\label{eq:H01-definition}
	H^{0,1}(\R)
	:=
	\left\{u\in L^2(\R): x u\in L^2(\R)\right\},
	\qquad
	\norm{u}_{H^{0,1}}
	:=
	\norm{u}_{L^2}+\snorm{ x u}_{L^2}.
\end{equation}
No physical derivative is included in $H^{0,1}$. Its spatial weight becomes a derivative only after passage to the interaction-profile variable.

For $0\leq\delta\leq\delta_0$ and $t\geq1$, define
\begin{equation}\label{eq:Pdelta-definition}
	P_\delta(t):=
	\begin{cases}
		\displaystyle\frac{2}{\delta}\bigl(1-t^{-\delta/2}\bigr),&\delta>0,\\[0.8em]
		\log t,&\delta=0.
	\end{cases}
\end{equation}
For an integer $m\geq0$ and $\delta>0$, define the $m$-th truncated clock by
\begin{equation}\label{eq:Pdelta-m-definition}
	P_\delta^{(m)}(t)
	:=
	\sum_{k=0}^{m}
	\frac{(-1)^k}{2^k(k+1)!}\,
	\delta^k(\log t)^{k+1}.
\end{equation}
For a profile $G$, define the frozen nonlinear phase
\begin{equation}\label{eq:frozen-phase-definition}
	\Phi_\delta[G](t,v):=\kappa 2^{-1-\delta/2} \abs{G(v)}^{2+\delta}P_\delta(t).
\end{equation}
We write $o_{L^p}(1)$ for a quantity converging to zero in the indicated $L^p$ space. The underlying variable is displayed when needed.

For $t\geq1$, define the self-similar amplitude by
\begin{equation}\label{eq:self-similar-amplitude-definition}
	u_\delta(t,x)=(2t)^{-1/2}\ee^{\frac{\ii x^2}{4t}}
	a_\delta\left(t,\frac{x}{2t}\right).
\end{equation}
This map is unitary from $L_v^2$ to $L_x^2$.

\subsection{The nonlinear clock}
\label{subsec:exact-phase-clock}

The following proposition is purely algebraic and does not use the PDE theory.

\begin{proposition}
	\label{prop:phase-clock-trichotomy}
	Let $\delta_n\to0$ and $t_n\to\infty$.  If $\delta_n\log t_n\to0$, then
	\begin{equation}\label{eq:clock-regime-one}
		P_{\delta_n}(t_n)\sim\log t_n.
	\end{equation}
	If $\delta_n\log t_n\to\Lambda\in(0,\infty)$, then
	\begin{equation}\label{eq:clock-regime-two}
		P_{\delta_n}(t_n)
		\sim
		\frac{2(1-\ee^{-\Lambda/2})}{\Lambda}\log t_n.
	\end{equation}
	If $\delta_n\log t_n\to\infty$, then
	\begin{equation}\label{eq:clock-regime-three}
		P_{\delta_n}(t_n)\sim\frac{2}{\delta_n}.
	\end{equation}
\end{proposition}

\begin{proof}
	For $\delta>0$, setting $\lambda=\delta\log t$ gives
	\[
		P_\delta(t)=\frac{2}{\delta}(1-\ee^{-\lambda/2})
		=\frac{2(1-\ee^{-\lambda/2})}{\lambda}\;\log t\;.
	\]
	The conclusions follow by letting $\lambda$ tend respectively to $0$, $\Lambda$, and $+\infty$.
\end{proof}

\begin{remark}
	\label{rem:relative-absolute-phase}
	The conclusions of \Cref{prop:phase-clock-trichotomy} are relative leading-order statements.  When $\delta\log t$ is small,
	\begin{equation}\label{eq:Pdelta-small-lambda-expansion}
		P_\delta(t)=\log t-\frac{\delta}{4}(\log t)^2+O\bigl(\delta^2(\log t)^3\bigr).
	\end{equation}
	More generally, the entire expansion is
	\begin{equation}\label{eq:Pdelta-full-cubic-expansion}
		P_\delta(t)
		=
		\sum_{k=0}^{\infty}
		\frac{(-1)^k}{2^k(k+1)!}\,
		\delta^k(\log t)^{k+1}.
	\end{equation}
	Thus $\delta\log t\to0$ does not imply $P_\delta(t)-\log t=o(1)$.  The sufficient condition $\delta(\log t)^2\to0$ is the $m=0$ case of the general truncation condition
	\[
	\delta^{m+1}(\log t)^{m+2}\to0
	\]
	for replacing $P_\delta$ by $P_\delta^{(m)}$ with absolute $o(1)$ error.  This hierarchy is stated in physical space in \Cref{thm:main-time-scale-hierarchy} and proved from an elementary remainder estimate in \Cref{subsec:cubic-side-hierarchy}.
\end{remark}

\begin{remark}
	\label{rem:absolute-saturation}
	For $\delta>0$,
	\begin{equation}\label{eq:saturated-clock-error}
		\frac{2}{\delta}-P_\delta(t)=\frac{2}{\delta}t^{-\delta/2}.
	\end{equation}
	Hence an absolute $o(1)$ replacement of $P_\delta(t)$ by $2/\delta$ requires $t^{-\delta/2}/\delta\to0$, which is stronger than $\delta\log t\to\infty$.
\end{remark}

\subsection{Main physical-space theorems}
\label{subsec:main-physical-hierarchy}

\begin{theorem}
	\label{thm:main-time-scale-hierarchy}
	There exist $0<\delta_0<1$ and $\varepsilon_0>0$ such that the following holds.  Assume
	\[
	\norm{u_{\mathrm{in}}}_{H^{0,1}}
	\leq
	\varepsilon
	\leq
	\varepsilon_0,
	\]
	and, for $0\leq\delta\leq\delta_0$, let
	\[
	u_\delta\in C(\R;L_x^2)\cap L^4_{t,\mathrm{loc}}(\R;L_x^\infty)
	\]
	be the unique global mass-class solution of \eqref{eq:nls-delta}.  There exists
	$W_0\in H_v^{3/4}$ such that the following hold.
	
	Let $\delta_n\to0^+$ and $t_n\to\infty$. Fix an integer $m\geq0$. If
	\begin{equation}\label{eq:main-hierarchy-assumptions}
		\delta_n^{m+1}(\log t_n)^{m+2}\longrightarrow0,
	\end{equation}
	then
	\begin{equation}\label{eq:main-cubic-hierarchy-clock}
		P_{\delta_n}(t_n)-P_{\delta_n}^{(m)}(t_n)
		\longrightarrow0
	\end{equation}
	and
	\begin{equation}\label{eq:main-cubic-hierarchy-physical}
		\begin{aligned}
			u_{\delta_n}(t_n,x)
			=&
			(2t_n)^{-1/2}
			\ee^{\frac{\ii x^2}{4t_n}}
			\ee^{-\ii\kappa
				P_{\delta_n}^{(m)}(t_n)
				\frac12\abs{W_0({x/2t_n})}^2}
			W_0\left({x/2t_n}\right)
			+
			o_{L_x^2}(1).
		\end{aligned}
	\end{equation}
	For $m\geq1$ and $c>0$, the scale
	\begin{equation}\label{eq:main-cubic-scale}
		t_n
		=
		\exp\left((c+o(1))\,\delta_n^{-m/(m+1)}\right)
	\end{equation}
	satisfies \eqref{eq:main-hierarchy-assumptions}.  At this scale,
	\begin{equation}\label{eq:main-cubic-scale-contribution}
		\begin{gathered}
			\delta_n^m(\log t_n)^{m+1}\longrightarrow c^{m+1},\\
			\frac12\bigl(P_{\delta_n}^{(m)}(t_n)-P_{\delta_n}^{(m-1)}(t_n)\bigr)
			\abs{W_0}^2
			\longrightarrow
			\frac{(-1)^m c^{m+1}}{2^{m+1}(m+1)!}\abs{W_0}^2
			\quad\text{in }L_v^2.
		\end{gathered}
	\end{equation}
\end{theorem}

\begin{theorem}\label{thm:exact_clock}
	Under the small-data assumptions of \Cref{thm:main-time-scale-hierarchy}, the same profile $W_0$ and profiles $Z_0\in H_v^{3/4}$ and $\dot Z_0\in L_v^2$ satisfy the following assertions. Here $Z_0$ is the limit after removal of the evolving phase, $\abs{Z_0}=\abs{W_0}$, and $\dot Z_0$ is its first exponent derivative, as constructed in \Cref{prop:first-order-profile-expansion}. In particular,
	\[
	2\operatorname{Re}(\overline{Z_0}\dot Z_0)
	+\abs{Z_0}^2\left(\log\abs{Z_0}-\frac12\log2\right)
	\in L_v^2.
	\]
	For arbitrary joint sequences $\delta_n\to0^+$ and $t_n\to\infty$, set $\lambda_n=\delta_n\log t_n$. Then
	\begin{equation}\label{eq:main-exact-clock-physical}
		\begin{aligned}
			u_{\delta_n}(t_n,x)
			=&(2t_n)^{-1/2}\ee^{\frac{\ii x^2}{4t_n}}W_0(x/2t_n)\\
			&\times\exp\Biggl\{-\ii\kappa\Biggl[
			\frac12P_{\delta_n}(t_n)\abs{W_0(x/2t_n)}^2\\
			&\hspace{5em}+(1-\ee^{-\lambda_n/2})\Bigl[
			2\operatorname{Re}\bigl(\overline{Z_0(x/2t_n)}\dot Z_0(x/2t_n)\bigr)\\
			&\hspace{7em}+\abs{Z_0(x/2t_n)}^2
			\left(\log\abs{Z_0(x/2t_n)}-\frac12\log2\right)
			\Bigr]\Biggr]\Biggr\}
			+o_{L_x^2}(1).
		\end{aligned}
	\end{equation}
	In \eqref{eq:main-exact-clock-physical}, the expression inside the outer phase brackets may be replaced, with the same $o_{L_x^2}(1)$ error, as follows. In the formulas below, the functions of $v$ are evaluated at $v=x/(2t_n)$.
	
	If $\lambda_n\to0$, then it may be replaced by
	\[
	\frac12 P_{\delta_n}(t_n)\abs{W_0}^2.
	\]
	
	If $\lambda_n\to\Lambda\in(0,\infty)$, then it may be replaced by
	\[
	\frac{1-\ee^{-\lambda_n/2}}{\delta_n}\abs{W_0}^2
	+(1-\ee^{-\Lambda/2})
	\left[
	2\operatorname{Re}(\overline{Z_0}\dot Z_0)
	+\abs{Z_0}^2
	\left(
	\log\abs{Z_0}-\frac12\log2
	\right)
	\right].
	\]
	
	If $\lambda_n\to\infty$, then it may be replaced by
	\[
	\frac{1-\ee^{-\lambda_n/2}}{\delta_n}\abs{W_0}^2
	+2\operatorname{Re}(\overline{Z_0}\dot Z_0)
	+\abs{Z_0}^2
	\left(
	\log\abs{Z_0}-\frac12\log2
	\right).
	\]
	
	All logarithmic products are defined by continuous extension at zero.
\end{theorem}

\subsection{Interpretation of the hierarchy}
\label{subsec:main-hierarchy-interpretation}

The times in \eqref{eq:main-cubic-scale} give the scale picture
\[
1
\ll
\ee^{\delta^{-1/2}}
\ll
\ee^{\delta^{-2/3}}
\ll
\ee^{\delta^{-3/4}}
\ll
\cdots
\ll
\ee^{\delta^{-m/(m+1)}}
\ll
\cdots
\ll
\ee^{O(1)/\delta}.
\]
They do not define mutually exclusive scattering regimes.  Every finite-$m$ scale lies inside the regime $\delta\log t\to0$.  What changes with time is the number of Taylor terms needed for absolute $o(1)$ phase accuracy:
\[
P_\delta^{(0)}
\rightsquigarrow
P_\delta^{(1)}
\rightsquigarrow
P_\delta^{(2)}
\rightsquigarrow
\cdots.
\]
The rigorous hypothesis is \eqref{eq:main-hierarchy-assumptions}, not merely an informal comparison with one of these times.  Since $m/(m+1)\uparrow1$, the finite Taylor scales accumulate at $t=\exp(O(1)/\delta)$.  At that transition scale the Taylor hierarchy is no longer finite, and the exact clock $P_\delta(t)=2(1-t^{-\delta/2})/\delta$ includes all its terms.  The scattering classification itself remains the three regimes stated in \Cref{thm:exact_clock}.

\subsection{Uniform asymptotics}
\label{subsec:uniform-asymptotics-asymptotics}

The following proposition gives an asymptotic profile for each exponent with a common error bound. The subsequent first-order expansion replaces these exponent-dependent profiles by the cubic profile and the explicit correction to the coefficient of the clock.

\begin{proposition}
	\label{prop:uniform-asymptotics}
	There exist $0<\delta_0<1$, $\varepsilon_0>0$, and $C>0$ such that the following holds. Assume
	\begin{equation}\label{eq:main-small-data}
		\norm{u_{\mathrm{in}}}_{H^{0,1}}\leq\varepsilon\leq\varepsilon_0.
	\end{equation}
	For every $0\leq\delta\leq\delta_0$, let
	\[
	u_\delta\in C(\R;L_x^2)\cap L^4_{t,\mathrm{loc}}(\R;L_x^\infty)
	\]
	be the unique global solution in the mass Strichartz class. Then there exist profiles
	$W_\delta\in H^{3/4}(\R)$ such that
	\begin{equation}\label{eq:uniform-asymptotics}
		\sup_{0\leq\delta\leq\delta_0}
		\norm{a_\delta(t)-\ee^{-\ii\Phi_\delta[W_\delta](t)}W_\delta}_{L_v^2}
		\leq C\varepsilon t^{-1/16},
		\qquad t\geq1.
	\end{equation}
	Moreover,
	\begin{equation}\label{eq:uniform-frozen-profile-Hs-bound}
		\sup_{0\leq\delta\leq\delta_0}
		\norm{W_\delta}_{H_v^{3/4}}
		\leq C\varepsilon.
	\end{equation}
\end{proposition}

\subsection{First-order expansion of the profiles}
\label{subsec:first-order-profile-expansion-main}

To turn \eqref{eq:uniform-asymptotics} into a joint-limit formula involving $W_0$, one must estimate
\[
	2^{-1-\delta/2} \abs{W_\delta}^{2+\delta}-2^{-1}\abs{W_0}^2.
\]
In the regime $\delta\log t\to0$, an $O(\delta)$ coefficient bound is compatible with $P_\delta(t)\sim\log t$.  In the transition and saturated regimes, however, $P_\delta(t)$ has size $\delta^{-1}$.  An $O(\delta)$ coefficient variation may therefore leave a finite, non-vanishing phase correction.  The following result identifies this correction in the $L^2$ topology needed for the physical-space asymptotics.

Phase removal produces a limiting profile $Z_\delta$. The convergent defect between the evolving and frozen phases is denoted by $\Gamma_\delta$, and $W_\delta=\ee^{-\ii\Gamma_\delta}Z_\delta$.

\begin{proposition}
	\label{prop:first-order-profile-expansion}
	Under the assumptions of \Cref{prop:uniform-asymptotics}, let $Z_\delta$, $\Gamma_\delta$, and $W_\delta$ be the profiles constructed in \Cref{sec:rescaled-amplitude-exact-phase}.  These profiles are uniformly bounded in $H^{3/4}(\R)$.  There exist
	\[
	\dot Z_0,\ \dot\Gamma_0,\ \dot W_0\in L^2(\R)
	\]
	such that, as $\delta\to0^+$,
	\begin{align}
		Z_\delta
		&=
		Z_0+\delta\dot Z_0+o_{L_v^2}(\delta),
		\label{eq:Z-first-order-main}\\
		\Gamma_\delta
		&=
		\Gamma_0+\delta\dot\Gamma_0+o_{L_v^2}(\delta),
		\label{eq:Gamma-first-order-main}\\
		W_\delta
		&=
		W_0+\delta\dot W_0+o_{L_v^2}(\delta).
		\label{eq:W-first-order-main}
	\end{align}
	Since $\abs{W_\delta}=\abs{Z_\delta}$, the coefficient satisfies
	\[
		2^{-1-\delta/2}\abs{W_\delta}^{2+\delta}
		=2^{-1-\delta/2}\abs{Z_\delta}^{2+\delta},
	\]
	and its first-order expansion is
	\begin{equation}\label{eq:coefficient-first-order-main}
			2^{-1-\delta/2}\abs{W_\delta}^{2+\delta}
			=\frac12\abs{W_0}^2
			+\frac{\delta}{2}\left[
			2\operatorname{Re}(\overline{Z_0}\dot Z_0)
			+\abs{Z_0}^2\left(\log\abs{Z_0}-\frac12\log2\right)
			\right]+o_{L_v^2}(\delta).
	\end{equation}
	Here $\abs{Z_0}^2\log\abs{Z_0}$ is defined to be zero on the zero set of $Z_0$.
\end{proposition}

Together, \Cref{prop:uniform-asymptotics,prop:first-order-profile-expansion} provide the PDE estimates and coefficient expansion for both main theorems.  

\section{Uniform preliminary estimates}
\label{sec:uniform-preliminary-estimates}

We prepare the uniform estimates for the phase removal in the next section. The mass and spatial-weight bounds give the initial profile regularity, and the long-time bounds will be completed by the coupled bootstrap there.

We use $H^{3/4}$ regularity to obtain bounded asymptotic profiles. The scattering statements themselves remain in $L^2$.

\subsection{Estimates for the nonlinearity}
\label{subsec:uniform-nonlinear-calculus}

For $0\leq\delta\leq\delta_0<1$, we estimate the nonlinearity directly. All derivatives below are real derivatives on $\C\simeq\R^2$.

\begin{lemma}
	\label{lem:uniform-tame-estimate}
	Uniformly for $0\leq\delta\leq\delta_0$,
	\begin{equation}\label{eq:uniform-pointwise-nonlinearity-derivatives}
		\abs{D_z^k(\abs{z}^{2+\delta}z)}
		\leq
		C\abs{z}^{3+\delta-k},
		\qquad
		k=1,2.
	\end{equation}
	Consequently, if $g\in H^1(\R)$, then
	\begin{equation}\label{eq:uniform-H1-tame}
		\snorm{\abs{g}^{2+\delta}g}_{H^1}
		\leq
		C\snorm{g}_{L^\infty}^{2+\delta}\snorm{g}_{H^1}.
	\end{equation}
	Moreover, whenever $\abs{z}+\abs{w}\leq1$,
	\begin{align}
		\abs{\abs{z}^{2+\delta}z-\abs{w}^{2+\delta}w}
		&\leq
		C(\abs{z}+\abs{w})^{2+\delta}\abs{z-w},
		\label{eq:uniform-nonlinearity-Lipschitz}\\
		\abs{D_z(\abs{z}^{2+\delta}z)-\left.D_z(\abs{z}^{2+\delta}z)\right|_{z=w}}
		&\leq
		C(\abs{z}+\abs{w})^{1+\delta}\abs{z-w}.
		\label{eq:uniform-nonlinearity-derivative-Lipschitz}
	\end{align}
\end{lemma}

\begin{proof}
	The map $z\mapsto\abs{z}^{2+\delta}z$ is homogeneous of degree $3+\delta$, and its real derivatives have coefficients polynomial in $\delta$.  This gives \eqref{eq:uniform-pointwise-nonlinearity-derivatives}, uniformly on $[0,\delta_0]$.  The estimates extend continuously at $z=0$.
	The chain rule gives
	\[
	\norm{\partial_v(\abs{g}^{2+\delta}g)}_{L^2}
	\leq
	C\norm{g}_{L^\infty}^{2+\delta}
	\norm{\partial_vg}_{L^2},
	\]
	which proves \eqref{eq:uniform-H1-tame}.  The mean-value formula and \eqref{eq:uniform-pointwise-nonlinearity-derivatives} prove \eqref{eq:uniform-nonlinearity-Lipschitz} and \eqref{eq:uniform-nonlinearity-derivative-Lipschitz}.
\end{proof}

\subsection{The self-similar equation}
\label{subsec:preliminary-amplitude-identity}

\begin{lemma}
	\label{lem:exact-self-similar-equation}
	The amplitude in \eqref{eq:self-similar-amplitude-definition} satisfies
	\begin{equation}\label{eq:exact-amplitude-equation}
		\ii\partial_ta_\delta
		+
		\frac1{4t^2}\partial_v^2a_\delta
		=
		\kappa 2^{-1-\delta/2} t^{-1-\delta/2}\abs{a_\delta}^{2+\delta}a_\delta.
	\end{equation}
\end{lemma}

\begin{proof}
	Put $v=x/(2t)$.  A direct differentiation of
	\[
	u_\delta(t,x)
	=
	(2t)^{-1/2}
	\ee^{\frac{\ii x^2}{4t}}a_\delta(t,v)
	\]
	gives
	\[
	\left(\ii\partial_t+\partial_x^2\right)u_\delta
	=
	(2t)^{-1/2}\ee^{\frac{\ii x^2}{4t}}
	\left(
	\ii\partial_ta_\delta
	+
	\frac1{4t^2}\partial_v^2a_\delta
	\right).
	\]
	Moreover,
	\[
	\abs{u_\delta}^{2+\delta}u_\delta
	=
	(2t)^{-3/2-\delta/2}
	\ee^{\frac{\ii x^2}{4t}}\abs{a_\delta}^{2+\delta}a_\delta.
	\]
	Dividing by the common factor gives \eqref{eq:exact-amplitude-equation}.  The evolution form follows immediately.
\end{proof}

\subsection{The small propagator}
\label{subsec:small-propagator}

For $t\geq1$, define
\begin{equation}\label{eq:small-propagator-definition}
	S(t)
	:=
	\exp\left(\frac{\ii}{4t}\partial_v^2\right).
\end{equation}

\begin{lemma}
	\label{lem:small-propagator-estimates}
	For every $f\in H^1(\R)$ and $t\geq1$,
	\begin{align}
		\snorm{(S(t)^{\pm1}-I)f}_{L^2}
		&\leq
		Ct^{-1/2}\snorm{f}_{H^1},
		\label{eq:small-propagator-L2}\\
		\snorm{(S(t)^{\pm1}-I)f}_{L^\infty}
		&\leq
		Ct^{-1/4}\snorm{f}_{H^1}.
		\label{eq:small-propagator-Linfty}
	\end{align}
\end{lemma}

\begin{proof}
	The Fourier multiplier of $S(t)-I$ is $\ee^{-\ii\xi^2/(4t)}-1$.  Since
	\[
	\abs{\ee^{\ii y}-1}\leq2\abs{y}^{1/2},
	\]
	Plancherel's theorem gives
	\[
	\snorm{(S(t)-I)f}_{L^2}
	\leq
	Ct^{-1/2}\snorm{\partial_vf}_{L^2}.
	\]
	The same proof applies to $S(t)^{-1}-I$, proving \eqref{eq:small-propagator-L2}.  If $g=(S(t)^{\pm1}-I)f$, then
	\[
	\snorm{g}_{L^2}
	\leq
	Ct^{-1/2}\norm{f}_{H^1},
	\qquad
	\snorm{\partial_vg}_{L^2}
	\leq
	2\snorm{\partial_vf}_{L^2}.
	\]
	The one-dimensional Gagliardo--Nirenberg inequality
	\[
	\snorm{g}_{L^\infty}
	\leq
	C\snorm{g}_{L^2}^{1/2}
	\snorm{\partial_vg}_{L^2}^{1/2}
	\]
	now proves \eqref{eq:small-propagator-Linfty}.
\end{proof}

\subsection{Uniform mass and profile-weight bounds}
\label{subsec:uniform-global-weighted-estimates}

Let
\begin{equation}\label{eq:interaction-profile-definition}
	f_\delta(t)
	:=
	\ee^{-\ii t\partial_x^2}u_\delta(t),
\end{equation}
and
\begin{equation}\label{eq:galilean-vector-field}
	J(t)
	:=
	x+2\ii t\partial_x
	=
	\ee^{\ii t\partial_x^2}x\ee^{-\ii t\partial_x^2}.
\end{equation}

\begin{proposition}
	\label{prop:uniform-preliminary-estimates}
	There exist $\varepsilon_0>0$, $C>0$, and $C_0>0$ such that the following holds. If
	\begin{equation}\label{eq:preliminary-smallness}
		\norm{u_{\mathrm{in}}}_{H^{0,1}}
		\leq
		\varepsilon
		\leq
		\varepsilon_0,
	\end{equation}
	then, for every $0\leq\delta\leq\delta_0$ and either sign $\kappa$, \eqref{eq:nls-delta} has a unique global solution
	\[
	u_\delta
	\in
	C(\R;L_x^2)\cap L^4_{t,\mathrm{loc}}(\R;L_x^\infty),
	\qquad
	J(t)u_\delta\in C([0,1];L_x^2).
	\]
	Its mass is conserved, and on $[0,1]$ one has
	\begin{equation}\label{eq:finite-time-J-bound}
		\sup_{0\leq\delta\leq\delta_0}
		\left(
		\norm{u_\delta}_{L_t^\infty L_x^2\cap L_t^4L_x^\infty([0,1])}
		+
		\norm{J(t)u_\delta}_{L_t^\infty L_x^2\cap L_t^4L_x^\infty([0,1])}
		\right)
		\leq C\varepsilon.
	\end{equation}
	For $t\geq1$,
	\begin{align}
		\sup_{0\leq\delta\leq\delta_0}
		\norm{u_\delta(t)}_{L_x^2}
		&\leq
		\varepsilon,
		\label{eq:uniform-global-mass}\\
		\sup_{0\leq\delta\leq\delta_0}
		\norm{a_\delta(t)}_{H_v^1}
		&\leq
		C\varepsilon t^{C_0\varepsilon^2},
		\label{eq:uniform-amplitude-H1}\\
		\sup_{0\leq\delta\leq\delta_0}
		\norm{a_\delta(t)}_{L_v^\infty}
		&\leq
		C\varepsilon,
		\label{eq:uniform-amplitude-Linfty}\\
		\sup_{0\leq\delta\leq\delta_0}
		\norm{u_\delta(t)}_{L_x^\infty}
		&\leq
		C\varepsilon t^{-1/2},
		\label{eq:uniform-physical-decay}\\
		\sup_{0\leq\delta\leq\delta_0}
		\norm{x f_\delta(t)}_{L_x^2}
		&\leq
		C\varepsilon t^{C_0\varepsilon^2}.
		\label{eq:uniform-one-weight-profile}
	\end{align}
\end{proposition}

\begin{proof}
	We first construct the solution in a mass Strichartz class, uniformly in the exponent.  On an interval $I\ni t_0$, the one-dimensional endpoint estimate for
	$(\ii\partial_t+\partial_x^2)w=G$ is
	\begin{equation}\label{eq:endpoint-Strichartz}
		\norm{w}_{L_t^\infty L_x^2\cap L_t^4L_x^\infty(I)}
		\leq
		C\left(
		\norm{w(t_0)}_{L_x^2}
		+
		\norm{G}_{L_t^1L_x^2(I)}
		\right).
	\end{equation}
	Set
	\[
	X(I):=L_t^\infty L_x^2(I)\cap L_t^4L_x^\infty(I).
	\]
	For $\abs{I}\leq1$, H\"older's inequality gives
	\begin{equation}\label{eq:mass-subcritical-nonlinearity}
		\begin{aligned}
			\norm{\abs{u}^{2+\delta}u}_{L_t^1L_x^2(I)}
			&\leq
			\norm{u}_{L_t^\infty L_x^2(I)}
			\int_I\norm{u(t)}_{L_x^\infty}^{2+\delta}\,\dd t\\
			&\leq
			\abs{I}^{(2-\delta)/4}
			\norm{u}_{L_t^\infty L_x^2(I)}
			\norm{u}_{L_t^4L_x^\infty(I)}^{2+\delta},
		\end{aligned}
	\end{equation}
	Since $\delta_0<1$, the time margin satisfies $(2-\delta)/4\geq(2-\delta_0)/4>0$.  The pointwise mean-value bound
	\[
	\abs{\abs{z}^{2+\delta}z-\abs{w}^{2+\delta}w}
	\leq
	C_{\delta_0}
	(\abs{z}^{2+\delta}+\abs{w}^{2+\delta})\abs{z-w}
	\]
	and the same calculation give
	\begin{equation}\label{eq:mass-subcritical-Lipschitz}
		\norm{\abs{u}^{2+\delta}u-\abs{w}^{2+\delta}w}_{L_t^1L_x^2(I)}
		\leq
		C\abs{I}^{(2-\delta)/4}
		(\norm{u}_{X(I)}^{2+\delta}+\norm{w}_{X(I)}^{2+\delta})
		\norm{u-w}_{X(I)}.
	\end{equation}
	Thus the Duhamel map is a contraction on each interval of length one whenever the initial mass is sufficiently small, with constants uniform in $\delta$ and independent of the sign of $\kappa$.  For Schwartz data, gauge invariance gives mass conservation.  Approximation in $L^2$ and \eqref{eq:mass-subcritical-Lipschitz} pass both the solution and mass conservation to rough data.  The conserved small mass permits iteration on successive unit intervals in both time directions and proves the asserted global solution class and \eqref{eq:uniform-global-mass}.  In particular, no energy or focusing coercivity is used.
	
	We next propagate the spatial weight covariantly on $[0,1]$.  The vector field in \eqref{eq:galilean-vector-field} commutes with $\ii\partial_t+\partial_x^2$.  Its action on the nonlinearity is
	\begin{equation}\label{eq:J-nonlinearity-identity}
		J(t)\bigl(\abs{u}^{2+\delta}u\bigr)
		=
		\left(1+\frac{2+\delta}{2}\right)\abs{u}^{2+\delta}J(t)u
		-
		\frac{2+\delta}{2}\abs{u}^{\delta}u^2\overline{J(t)u}.
	\end{equation}
	In particular,
	\begin{equation}\label{eq:J-nonlinearity-bound}
		\abs{J(t)\bigl(\abs{u}^{2+\delta}u\bigr)}
		\leq
		C_{\delta_0}\abs{u}^{2+\delta}\abs{J(t)u}.
	\end{equation}
	For a Schwartz solution, put $V=J(t)u$.  Applying \eqref{eq:endpoint-Strichartz} to the equation for $V$ and then using \eqref{eq:J-nonlinearity-bound} and H\"older as in \eqref{eq:mass-subcritical-nonlinearity}, we obtain
	\begin{equation}\label{eq:J-Strichartz-bound}
		\norm{V}_{X([0,1])}
		\leq
		C\norm{xu_{\mathrm{in}}}_{L_x^2}
		+
		C\norm{u}_{L_t^4L_x^\infty([0,1])}^{2+\delta}
		\norm{V}_{L_t^\infty L_x^2([0,1])}.
	\end{equation}
	The previously established bound $\norm{u}_{X([0,1])}\leq C\varepsilon$ allows the last term to be absorbed after decreasing $\varepsilon_0$, uniformly in $\delta$.
	
	To justify the vector-field identity for $H^{0,1}$ data, let $u_{\mathrm{in},n}\to u_{\mathrm{in}}$ be a Schwartz approximation in $H^{0,1}$.  Then $u_n\to u$ in $X([0,1])$.  If
	\[
	G_\delta(u,V)
	:=
	\left(1+\frac{2+\delta}{2}\right)\abs{u}^{2+\delta}V
	-
	\frac{2+\delta}{2}\abs{u}^{\delta}u^2\overline V,
	\]
	then
	\begin{equation*}
		\abs{G_\delta(u,V)-G_\delta(\widetilde u,\widetilde V)}
		\leq
		C(\abs{u}^{2+\delta}+\abs{\widetilde u}^{2+\delta})\abs{V-\widetilde V}
		+C(\abs{u}^{1+\delta}+\abs{\widetilde u}^{1+\delta})
		(\abs V+\abs{\widetilde V})\abs{u-\widetilde u}.
	\end{equation*}
	Placing $V,\widetilde V$ in $L_t^\infty L_x^2$ and the solution factors of total power $2+\delta$ in $L_t^4L_x^\infty$ yields the same positive time margin $(2-\delta)/4$.  Hence $J(t)u_n$ is Cauchy in $X([0,1])$.  Its limit $V$ equals $J(t)u$ distributionally, and its Duhamel source belongs to $L_t^1L_x^2$. Consequently $V\in C([0,1];L_x^2)$ and $V(0)=xu_{\mathrm{in}}$.  This proves \eqref{eq:finite-time-J-bound} and justifies the vector-field calculation for $H^{0,1}$ data.
	
	At $t=1$,
	\[
	a_\delta(1,v)
	=
	2^{1/2}\ee^{-\ii v^2}u_\delta(1,2v),
	\]
	and direct differentiation gives
	\[
	\partial_va_\delta(1,v)
	=
	-\ii2^{1/2}\ee^{-\ii v^2}
	\left(J(1)u_\delta\right)(1,2v).
	\]
	Hence
	\begin{equation}\label{eq:amplitude-at-one-H1}
		\sup_{0\leq\delta\leq\delta_0}
		\norm{a_\delta(1)}_{H_v^1}
		\leq
		C\varepsilon.
	\end{equation}
	
	Starting at $t=1$, Picard iteration for the amplitude equation \eqref{eq:exact-amplitude-equation} applies in $C_tH_v^1$, since its coefficient is smooth for $t\geq1$ and \eqref{eq:uniform-H1-tame} is an estimate in the $v$ variable.  Transforming this solution back to physical variables produces a solution in the mass Strichartz class with the same time-one datum. Uniqueness in that class identifies it with the global solution constructed above.  This restart supplies profile regularity and does not assert physical $H_x^1$ smoothing.
	
	The skew-adjoint linear term in \eqref{eq:exact-amplitude-equation} and \eqref{eq:uniform-H1-tame} give the conditional differential inequality
	\begin{equation}\label{eq:conditional-amplitude-H1}
		\frac{\dd}{\dd t}\norm{a_\delta(t)}_{H_v^1}
		\leq
		Ct^{-1}\norm{a_\delta(t)}_{L_v^\infty}^{2+\delta}
		\norm{a_\delta(t)}_{H_v^1}.
	\end{equation}
	Proposition~\ref{prop:phase-removed-convergence} closes \eqref{eq:conditional-amplitude-H1} through a coupled bootstrap for $\norm{A_\delta}_{L_v^\infty}$ and $\norm{a_\delta}_{H_v^1}$.  This gives \eqref{eq:uniform-amplitude-H1} and \eqref{eq:uniform-amplitude-Linfty}.  The scaling \eqref{eq:self-similar-amplitude-definition} then gives \eqref{eq:uniform-physical-decay}.  Finally,
	\[
	\norm{\partial_va_\delta(t)}_{L^2}
	=
	\norm{J(t)u_\delta(t)}_{L^2}
	=
	\norm{x f_\delta(t)}_{L^2},
	\]
	so \eqref{eq:uniform-one-weight-profile} follows from \eqref{eq:uniform-amplitude-H1}.
\end{proof}

\section{Phase removal and asymptotic profile}
\label{sec:rescaled-amplitude-exact-phase}

We separate the leading nonlinear phase from an integrable remainder, close the uniform bootstrap, and construct the asymptotic profiles in \Cref{prop:uniform-asymptotics}.

Throughout this section and \Cref{sec:quantitative-profile-continuity}, unadorned Sobolev and Lebesgue norms are taken in the profile variable $v$. Norms in the physical variable are marked by an $x$ subscript.

\subsection{The conjugated amplitude and its remainder}
\label{subsec:exact-small-propagator}

Define
\begin{equation}\label{eq:conjugated-amplitude-definition}
	A_\delta(t)
	:=
	S(t)a_\delta(t).
\end{equation}
We use the unitary Fourier transform
\[
\widehat g(\xi)
=
(2\pi)^{-1/2}\int_\R\ee^{-\ii x\xi}g(x)\,\dd x.
\]
Completing the square in the free Schr\"odinger representation and using
\eqref{eq:interaction-profile-definition} gives the profile identity
\begin{equation}\label{eq:exact-Fourier-profile-identity}
	A_\delta(t)
	=
	\ee^{-\ii\pi/4}\widehat f_\delta(t).
\end{equation}
Consequently, Plancherel's theorem, $\partial_\xi\widehat f=-\ii\widehat{xf}$, and the unitarity of $S(t)$ on $H_v^1$ give
\begin{equation}\label{eq:profile-derivative-weight-identity}
	\norm{\partial_va_\delta(t)}_{L_v^2}
	=
	\norm{\partial_vA_\delta(t)}_{L_v^2}
	=
	\norm{x f_\delta(t)}_{L_x^2}
	=
	\norm{J(t)u_\delta(t)}_{L_x^2}.
\end{equation}
Thus every $H_v^1$ input below comes from the spatial weight, not from a physical $H_x^1$ assumption.
Since
\[
\partial_tS(t)
=
-\frac{\ii}{4t^2}\partial_v^2S(t),
\]
\eqref{eq:exact-amplitude-equation} gives the equation
\begin{equation}\label{eq:conjugated-amplitude-equation}
	\partial_tA_\delta
	=
	-\ii\kappa 2^{-1-\delta/2}t^{-1-\delta/2}
	\left[
	\abs{A_\delta}^{2+\delta}A_\delta
	+
	\mathcal R_\delta^\sharp(t)
	\right],
\end{equation}
where
\begin{equation}\label{eq:conjugation-remainder-definition}
	\mathcal R_\delta(t,A)
	:=
	S(t)\bigl(\abs{S(t)^{-1}A}^{2+\delta}S(t)^{-1}A\bigr)-\abs{A}^{2+\delta}A,
	\qquad
	\mathcal R_\delta^\sharp(t)
	:=
	\mathcal R_\delta(t,A_\delta(t)).
\end{equation}

\begin{lemma}
	\label{lem:conjugation-remainder-estimates}
	Suppose on an interval $1\leq t\leq T$ that
	\[
	\norm{A_\delta(t)}_{L^\infty}\leq C\varepsilon,
	\qquad
	\norm{a_\delta(t)}_{H^1}\leq C\varepsilon t^{C_0\varepsilon^2}.
	\]
	Then, uniformly for $0\leq\delta\leq\delta_0$,
	\begin{align}
		\snorm{\mathcal R_\delta^\sharp(t)}_{L^2}
		&\leq
		C\varepsilon^3t^{-1/2+C_0 \varepsilon^2},
		\label{eq:conjugation-remainder-L2}\\
		\snorm{\mathcal R_\delta^\sharp(t)}_{H^1}
		&\leq
		C\varepsilon^3t^{C_0\varepsilon^2},
		\label{eq:conjugation-remainder-H1}\\
		\snorm{\mathcal R_\delta^\sharp(t)}_{H^{3/4}}
		&\leq
		C\varepsilon^3t^{-1/8+C_0 \varepsilon^2}.
		\label{eq:conjugation-remainder-Hs}
	\end{align}
\end{lemma}

\begin{proof}
	By \eqref{eq:small-propagator-Linfty},
	\[
	\norm{a_\delta-A_\delta}_{L^\infty}
	\leq
	Ct^{-1/4}\norm{A_\delta}_{H^1}
	\leq
	C\varepsilon t^{-1/4+C_0 \varepsilon^2}.
	\]
	After decreasing $\varepsilon_0$ so that $C_0 \varepsilon^2<1/4$, both $a_\delta$ and $A_\delta$ are bounded by $C\varepsilon$ in $L^\infty$.
	Decompose
	\begin{equation}\label{eq:conjugation-remainder-decomposition}
		\mathcal R_\delta^\sharp(t)
		=
		(S(t)-I)\bigl(\abs{a_\delta}^{2+\delta}a_\delta\bigr)
		+
		\abs{a_\delta}^{2+\delta}a_\delta-\abs{A_\delta}^{2+\delta}A_\delta.
	\end{equation}
	By \Cref{lem:small-propagator-estimates}, \eqref{eq:uniform-H1-tame}, and unitarity of $S(t)$ on $H^1$,
	\[
	\snorm{(S(t)-I)\bigl(\abs{a_\delta}^{2+\delta}a_\delta\bigr)}_{L^2}
	\leq
	Ct^{-1/2}\snorm{\abs{a_\delta}^{2+\delta}a_\delta}_{H^1}
	\leq
	C\varepsilon^2t^{-1/2}\norm{a_\delta}_{H^1}.
	\]
	Similarly, \eqref{eq:uniform-nonlinearity-Lipschitz} gives
	\[
	\snorm{\abs{a_\delta}^{2+\delta}a_\delta-\abs{A_\delta}^{2+\delta}A_\delta}_{L^2}
	\leq
	C\varepsilon^2\snorm{a_\delta-A_\delta}_{L^2}
	\leq
	C\varepsilon^2t^{-1/2}\norm{a_\delta}_{H^1}.
	\]
	Substituting the assumed $H^1$ bound proves \eqref{eq:conjugation-remainder-L2}.  On the other hand,
	\[
	\snorm{\mathcal R_\delta^\sharp(t)}_{H^1}
	\leq
	\snorm{\abs{a_\delta}^{2+\delta}a_\delta}_{H^1}
	+
	\snorm{\abs{A_\delta}^{2+\delta}A_\delta}_{H^1}
	\leq
	C\varepsilon^2\norm{a_\delta}_{H^1},
	\]
	which gives \eqref{eq:conjugation-remainder-H1}.  Interpolating between $L^2$ and $H^1$, with interpolation weights $1/4$ and $3/4$, gives
	\[
	\snorm{\mathcal R_\delta^\sharp(t)}_{H^{3/4}}
	\leq
	C
	\snorm{\mathcal R_\delta^\sharp(t)}_{L^2}^{1/4}
	\snorm{\mathcal R_\delta^\sharp(t)}_{H^1}^{3/4}
	\leq
	C\varepsilon^2t^{-1/8}\norm{a_\delta}_{H^1}.
	\]
The same substitution gives \eqref{eq:conjugation-remainder-Hs}.
	These estimates are linear in $\norm{a_\delta}_{H^1}$ and therefore
	preserve any smaller polynomial exponent available for that norm.
\end{proof}

\subsection{Phase removal and convergence}
\label{subsec:exact-phase-removal}

Define
\begin{align}
	\Theta_\delta(t,v)
	&:=
	\kappa\int_1^t
	2^{-1-\delta/2}\tau^{-1-\delta/2}
	\abs{A_\delta(\tau,v)}^{2+\delta}
	\,\dd\tau,
	\label{eq:Theta-delta-definition}\\
	b_\delta(t,v)
	&:=
	\ee^{\ii\Theta_\delta(t,v)}A_\delta(t,v).
	\label{eq:b-delta-definition}
\end{align}
The phase is real-valued.  By \eqref{eq:conjugated-amplitude-equation}, the leading nonlinear term cancels exactly and
\begin{equation}\label{eq:exact-b-equation}
	\partial_tb_\delta
	=
	-\ii\kappa 2^{-1-\delta/2}t^{-1-\delta/2}
	\ee^{\ii\Theta_\delta}
	\mathcal R_\delta^\sharp(t).
\end{equation}

\begin{proposition}
	\label{prop:phase-removed-convergence}
	After decreasing $\varepsilon_0$ so that
	\begin{equation}\label{eq:section-four-smallness}
		2C_0\varepsilon_0^2\leq\frac1{16},
	\end{equation}
	the bootstrap in \Cref{prop:uniform-preliminary-estimates} closes
	uniformly for $0\leq\delta\leq\delta_0$. Moreover,
	\begin{equation}\label{eq:b-time-derivative-Hs}
		\norm{\partial_tb_\delta(t)}_{H^{3/4}}
		\leq
		C\varepsilon^3t^{-1-1/16}.
	\end{equation}
	Consequently, there is $Z_\delta\in H^{3/4}(\R)$ such that
	\begin{align}
		\sup_{0\leq\delta\leq\delta_0}
		\norm{b_\delta(t)-Z_\delta}_{H^{3/4}}
		&\leq
		C\varepsilon^3t^{-1/16},
		\label{eq:b-Z-tail-Hs}\\
		\sup_{0\leq\delta\leq\delta_0}
		\norm{Z_\delta}_{H^{3/4}}
		&\leq
		C\varepsilon.
		\label{eq:Z-uniform-Hs}
	\end{align}
\end{proposition}

\begin{proof}
	We first improve the $H^1$ bound under the assumed $L^\infty$ bound. We then control $b_\delta$ in $H^{3/4}$ to improve the $L^\infty$ bound and close the bootstrap before constructing $Z_\delta$.
	
	By \eqref{eq:amplitude-at-one-H1} and the unitarity of $S(1)$,
	\begin{equation}\label{eq:coupled-bootstrap-initial-bound}
		\norm{a_\delta(1)}_{H^1}+\norm{A_\delta(1)}_{H^{3/4}}
		\leq C\varepsilon.
	\end{equation}
Constants $C$ in polynomial exponents below and in \Cref{sec:quantitative-profile-continuity} are fixed absolute constants, independent of $\delta,t,\varepsilon,C_0$, and may increase from line to line. The single constant $C_0$ is chosen larger than all the finitely many exponent constants obtained in these proofs; $\varepsilon_0$ is then decreased to satisfy \eqref{eq:section-four-smallness} and all earlier smallness conditions. Starting from \eqref{eq:coupled-bootstrap-initial-bound}, a standard continuity argument reduces the proof to improving bounds of the form
\begin{align}
\norm{A_\delta(t)}_{L^\infty}&\leq C\varepsilon,
\label{eq:A-Linfty-bootstrap}\\
\norm{a_\delta(t)}_{H^1}&\leq C\varepsilon t^{C_0\varepsilon^2}
\label{eq:a-H1-coupled-bootstrap}
\end{align}
on a maximal interval $[1,T]$. The constants are fixed independently of $\delta,t,\varepsilon$, sufficiently large in terms of the initial bound and Sobolev embedding that both inequalities are strict at $t=1$. Since $a_\delta=S(t)^{-1}A_\delta$, the small-propagator estimate gives
	\[
		\norm{a_\delta(t)}_{L^\infty}
		\leq\norm{A_\delta(t)}_{L^\infty}
		+Ct^{-1/4}\norm{a_\delta(t)}_{H^1}
		\leq C\varepsilon,
	\]
provided $C_0\varepsilon_0^2<1/4$. Substitution into \eqref{eq:conditional-amplitude-H1} yields
	\[
	\frac{\dd}{\dd t}\norm{a_\delta(t)}_{H^1}
	\leq C\varepsilon^2t^{-1}\norm{a_\delta(t)}_{H^1},
	\]
since $(C\varepsilon)^{2+\delta}\leq C\varepsilon^2$ uniformly in $\delta$ for $\varepsilon\leq1$. Gr\"onwall's inequality gives
	\begin{equation}\label{eq:a-H1-coupled-improvement}
		\norm{a_\delta(t)}_{H^1}
		\leq\norm{a_\delta(1)}_{H^1}t^{C\varepsilon^2}\leq C\varepsilon t^{C\varepsilon^2}.
	\end{equation}
The exponent constant here is absolute. Choosing $C_0$ larger than it, and using the strict initial bound, strictly improves \eqref{eq:a-H1-coupled-bootstrap}.
The proof of \Cref{lem:conjugation-remainder-estimates}, with \eqref{eq:a-H1-coupled-improvement} as input, now gives \eqref{eq:conjugation-remainder-L2}--\eqref{eq:conjugation-remainder-Hs} with actual exponents $C\varepsilon^2$.

We next estimate the phase. Since $2^{-1-\delta/2}t^{-\delta/2}\leq C$, mass conservation gives
	\[
		\norm{\Theta_\delta(t)}_{L^2}
		\leq
		C\int_1^t
		\tau^{-1}
		\norm{A_\delta(\tau)}_{L^\infty}^{1+\delta}
		\norm{A_\delta(\tau)}_{L^2}
		\,\dd\tau
		\leq
		C\varepsilon^2(1+\log t).
	\]
The real chain rule and \eqref{eq:a-H1-coupled-improvement} also give
	\[
		\norm{\partial_v\Theta_\delta(t)}_{L^2}
		\leq
		C\varepsilon^{2+\delta}
		\int_1^t\tau^{-1+C\varepsilon^2}\,\dd\tau
		\leq
		Ct^{C\varepsilon^2}.
	\]
Since
	\[
	\varepsilon^2(1+\log t)
	\leq
	Ct^{C\varepsilon^2},
	\]
we obtain
	\begin{equation}\label{eq:Theta-H1-growth}
		1+\norm{\Theta_\delta(t)}_{H^1}
		\leq
		Ct^{C\varepsilon^2}.
	\end{equation}
	
Since $3/4>1/2$, the fractional product estimate reads
	\begin{equation}\label{eq:fractional-product-estimate}
		\norm{fg}_{H^{3/4}}
		\leq
		C\left(
		\norm{f}_{L^\infty}\norm{g}_{H^{3/4}}
		+
		\norm{f}_{H^{3/4}}\norm{g}_{L^\infty}
		\right).
	\end{equation}
In particular, $H^{3/4}(\R)$ is an algebra. Since $\Theta_\delta$ is real,
	\[
	\abs{\ee^{\ii\Theta_\delta}-1}\leq\abs{\Theta_\delta},
	\qquad
	\partial_v(\ee^{\ii\Theta_\delta}-1)
	=
	\ii\ee^{\ii\Theta_\delta}\partial_v\Theta_\delta.
	\]
Thus
	\[
	\norm{\ee^{\ii\Theta_\delta}-1}_{H^{3/4}}
	\leq
	\norm{\ee^{\ii\Theta_\delta}-1}_{H^1}
	\leq
	C\norm{\Theta_\delta}_{H^1}.
	\]
Applying \eqref{eq:fractional-product-estimate} to
	\[
	\ee^{\ii\Theta_\delta}\mathcal R_\delta^\sharp
	=
	\mathcal R_\delta^\sharp
	+
	(\ee^{\ii\Theta_\delta}-1)\mathcal R_\delta^\sharp
	\]
gives
	\begin{equation}\label{eq:phase-multiplier-Hs}
		\snorm{\ee^{\ii\Theta_\delta}\mathcal R_\delta^\sharp}_{H^{3/4}}
		\leq
		C\left(1+\norm{\Theta_\delta}_{H^1}\right)
		\snorm{\mathcal R_\delta^\sharp}_{H^{3/4}}.
	\end{equation}
	Here the second term in \eqref{eq:fractional-product-estimate} is controlled by
	\[
	\norm{\mathcal R_\delta^\sharp}_{L^\infty}
	\le C
	\norm{\mathcal R_\delta^\sharp}_{H^{3/4}}.
	\]
Use the actual polynomial bounds underlying \eqref{eq:conjugation-remainder-Hs} and \eqref{eq:Theta-H1-growth}. Since $2^{-1-\delta/2}t^{-\delta/2}\leq C$, \eqref{eq:exact-b-equation} gives
	\[
	\norm{\partial_tb_\delta(t)}_{H^{3/4}}
	\leq
	C\varepsilon^3
	t^{-1-1/8+C\varepsilon^2}.
	\]
The exponent constant in this last bound includes the sum of the constants from the remainder and the phase multiplier. Once it is dominated by $C_0$, \eqref{eq:section-four-smallness} gives $C\varepsilon^2\leq C_0\varepsilon_0^2\leq1/32<1/16$. This proves \eqref{eq:b-time-derivative-Hs}. Integrating gives
	\[
	\norm{b_\delta(t)-b_\delta(T)}_{H^{3/4}}
	\leq
	C\varepsilon^3t^{-1/16},
	\qquad
	T\geq t.
	\]
	
Furthermore,
	\[
	\norm{b_\delta(t)}_{H^{3/4}}
	\leq\norm{A_\delta(1)}_{H^{3/4}}+C\varepsilon^3.
	\]
Since $\abs{A_\delta}=\abs{b_\delta}$, Sobolev embedding and the initial bound now strictly improve the assumed $L^\infty$ bound \eqref{eq:A-Linfty-bootstrap}, after decreasing $\varepsilon_0$ if necessary. Both bounds have therefore improved. If the maximal endpoint $T$ were finite, continuity would extend the bounds beyond $T$, a contradiction. Hence $T=\infty$, uniformly in $\delta$.

The tail estimate now constructs $Z_\delta$ in $H^{3/4}$ and proves \eqref{eq:b-Z-tail-Hs}. Letting $t\to\infty$ in the preceding uniform $H^{3/4}$ bound gives \eqref{eq:Z-uniform-Hs}. This also completes the deferred $L^\infty$ and $H^1$ assertions in \Cref{prop:uniform-preliminary-estimates}.
\end{proof}

\subsection{Phase defect and asymptotic profile}
\label{subsec:phase-defect-frozen-profile}

Define
\begin{equation}\label{eq:phase-defect-definition}
	D_\delta(t)
	:=
	\Theta_\delta(t)
	-
	\kappa 2^{-1-\delta/2} \abs{Z_\delta}^{2+\delta}P_\delta(t).
\end{equation}
Because $\Theta_\delta$ is real,
\begin{equation}\label{eq:modulus-A-b}
	\abs{A_\delta}=\abs{b_\delta}.
\end{equation}
Consequently,
\begin{equation}\label{eq:phase-defect-integral}
	D_\delta(t)
	=
	\kappa 2^{-1-\delta/2}
	\int_1^t
	\tau^{-1-\delta/2}
	\left(
	\abs{b_\delta(\tau)}^{2+\delta}
	-
	\abs{Z_\delta}^{2+\delta}
	\right)
	\,\dd\tau.
\end{equation}

\begin{lemma}
\label{lem:frozen-density-difference}
Uniformly for $0\leq\delta\leq\delta_0$,
	\begin{equation}\label{eq:frozen-density-difference-Hs}
		\norm{
			\abs{b_\delta(t)}^{2+\delta}
			-
			\abs{Z_\delta}^{2+\delta}
		}_{H^{3/4}}
		\leq
		C\varepsilon^2t^{-1/16}.
	\end{equation}
\end{lemma}

\begin{proof}
Uniformly in $\delta$,
	\[
	\abs{D_z(\abs{z}^{2+\delta})}
	\leq
	C\abs{z}^{1+\delta},
	\qquad
	\abs{D_z^2(\abs{z}^{2+\delta})}
	\leq
	C\abs{z}^{\delta}
	\leq C
	\]
on the bounded range of $b_\delta$ and $Z_\delta$.
The mean-value formula, \eqref{eq:fractional-product-estimate}, and \eqref{eq:b-Z-tail-Hs} give
	\[
	\norm{\abs{b_\delta}^{2+\delta}-\abs{Z_\delta}^{2+\delta}}_{H^{3/4}}
	\leq
	C\varepsilon^{1+\delta}
	\norm{b_\delta-Z_\delta}_{H^{3/4}}
	\leq
	C\varepsilon^2t^{-1/16}.
	\]
	At $\delta=0$, $D_z^2(\abs{z}^2)$ is bounded, so the estimate remains uniform at the endpoint.
\end{proof}

\begin{proposition}
\label{prop:phase-defect-convergence}
The real-valued limit
	\begin{equation}\label{eq:Gamma-delta-definition}
		\Gamma_\delta
		:=
		\kappa 2^{-1-\delta/2}
		\int_1^\infty
		t^{-1-\delta/2}
		\left(
		\abs{b_\delta(t)}^{2+\delta}
		-
		\abs{Z_\delta}^{2+\delta}
		\right)
		\,\dd t
	\end{equation}
exists in $H^{3/4}(\R)$, uniformly in $\delta$. Moreover,
	\begin{align}
		\norm{D_\delta(t)-\Gamma_\delta}_{H^{3/4}}
		&\leq
		C\varepsilon^2t^{-1/16-\delta/2},
		\label{eq:phase-defect-tail-Hs}\\
		\sup_{0\leq\delta\leq\delta_0}
		\norm{\Gamma_\delta}_{H^{3/4}}
		&\leq
		C\varepsilon^2.
		\label{eq:Gamma-uniform-Hs}
	\end{align}
In particular,
	\begin{equation}\label{eq:phase-defect-tail-Linfty}
		\norm{D_\delta(t)-\Gamma_\delta}_{L^\infty}
		\leq
		C\varepsilon^2t^{-1/16}.
	\end{equation}
\end{proposition}

\begin{proof}
\Cref{lem:frozen-density-difference} gives
	\[
	\norm{D_\delta(t)-\Gamma_\delta}_{H^{3/4}}
	\leq
	C\varepsilon^2
	\int_t^\infty
	\tau^{-1-\delta/2-1/16}\,\dd\tau
	\leq
	C\varepsilon^2t^{-1/16-\delta/2}.
	\]
This proves convergence uniformly down to $\delta=0$ and \eqref{eq:phase-defect-tail-Hs}. Taking $t=1$, where $D_\delta(1)=0$, gives \eqref{eq:Gamma-uniform-Hs}. The embedding $H^{3/4}(\R)\hookrightarrow L^\infty(\R)$ gives \eqref{eq:phase-defect-tail-Linfty}.
\end{proof}

\begin{lemma}
\label{lem:frozen-profile-properties}
Set
\begin{equation}\label{eq:frozen-profile-definition}
W_\delta:=\ee^{-\ii\Gamma_\delta}Z_\delta.
\end{equation}
Then
\begin{align}
\abs{W_\delta}&=\abs{Z_\delta},
\label{eq:frozen-profile-modulus}\\
\sup_{0\leq\delta\leq\delta_0}\norm{W_\delta}_{H^{3/4}}
&\leq C\varepsilon.
\label{eq:frozen-profile-Hs-bound}
\end{align}
	Moreover,
	\begin{equation}\label{eq:frozen-profile-phase-identity}
		\Phi_\delta[W_\delta](t)
		=
		\kappa 2^{-1-\delta/2} \abs{Z_\delta}^{2+\delta}P_\delta(t).
	\end{equation}
\end{lemma}

\begin{proof}
The modulus identity follows because $\Gamma_\delta$ is real. Since $0<3/4<1$, the Lipschitz composition estimate gives
	\[
	\norm{\ee^{-\ii\Gamma_\delta}-1}_{H^{3/4}}
	\leq
	C\norm{\Gamma_\delta}_{H^{3/4}}.
	\]
Using \eqref{eq:fractional-product-estimate}, \eqref{eq:Z-uniform-Hs}, and \eqref{eq:Gamma-uniform-Hs}, we obtain \eqref{eq:frozen-profile-Hs-bound}. The phase identity follows from the definition of $\Phi_\delta$, \eqref{eq:frozen-profile-modulus}, and \eqref{eq:frozen-profile-definition}.
\end{proof}

\subsection{Uniform profile asymptotics}
\label{subsec:uniform-frozen-profile-asymptotics}

\begin{proposition}
	\label{prop:uniform-frozen-profile-asymptotics}
	Under the assumptions of \Cref{prop:uniform-preliminary-estimates},
	\begin{equation}\label{eq:uniform-frozen-profile-asymptotics}
		\sup_{0\leq\delta\leq\delta_0}
		\norm{
			a_\delta(t)
			-
			\ee^{-\ii\Phi_\delta[W_\delta](t)}W_\delta
		}_{L^2}
		\leq
		C\varepsilon t^{-1/16},
		\qquad
		t\geq1.
	\end{equation}
\end{proposition}

\begin{proof}
	Since $a_\delta=S(t)^{-1}A_\delta$,
	\[
	\norm{a_\delta(t)-A_\delta(t)}_{L^2}
	\leq
	Ct^{-1/2}\norm{A_\delta(t)}_{H^1}
	\leq
	C\varepsilon t^{-1/2+C_0\varepsilon^2}.
	\]
	By the choice of $\varepsilon_0$, this term decays faster than $t^{-1/16}$.
	
	By \eqref{eq:phase-defect-definition}, \eqref{eq:b-delta-definition}, and \eqref{eq:frozen-profile-phase-identity},
	\begin{align*}
		\snorm{
			A_\delta(t)
			-
			\ee^{-\ii\Phi_\delta[W_\delta](t)}W_\delta
		}_{L^2}
		&=
		\snorm{
			\ee^{-\ii D_\delta(t)}b_\delta(t)
			-
			\ee^{-\ii\Gamma_\delta}Z_\delta
		}_{L^2}\\
		&\leq
		\snorm{b_\delta(t)-Z_\delta}_{L^2}
		+
		\snorm{D_\delta(t)-\Gamma_\delta}_{L^\infty}
		\snorm{Z_\delta}_{L^2}.
	\end{align*}
	The first term is bounded by \eqref{eq:b-Z-tail-Hs}, and the second by \eqref{eq:phase-defect-tail-Linfty} and \eqref{eq:Z-uniform-Hs}. This proves \eqref{eq:uniform-frozen-profile-asymptotics}.
\end{proof}

\begin{proof}[Proof of \Cref{prop:uniform-asymptotics}]
	The global solutions and uniform bounds follow from \Cref{prop:uniform-preliminary-estimates,prop:phase-removed-convergence}. The asymptotic profiles are defined by \eqref{eq:frozen-profile-definition}, their $H^{3/4}$ bound is \eqref{eq:frozen-profile-Hs-bound}, and \Cref{prop:uniform-frozen-profile-asymptotics} gives the uniform $L^2$ asymptotic estimate.
\end{proof}

\section{First-order expansion in the exponent}
\label{sec:quantitative-profile-continuity}

This section proves \Cref{prop:first-order-profile-expansion} by following the first-order dependence on the exponent through the construction
\[
u_\delta\longrightarrow a_\delta\longrightarrow A_\delta
\longrightarrow b_\delta\longrightarrow (Z_\delta,\Gamma_\delta)
\longrightarrow W_\delta
\longrightarrow 2^{-1-\delta/2}\abs{W_\delta}^{2+\delta}.
\]
The quotients $q_\delta$, $h_\delta$, $k_\delta$, and $g_\delta$ below are the parameter quotients of the already constructed objects $u_\delta$, $a_\delta$, $A_\delta$, and $b_\delta$, respectively. They are used where an evolution equation or a substantial estimate must be carried out. The $H_v^1$ estimate for the amplitude quotient passes through conjugation to an integrable $L_v^2$ equation for the quotient of $b_\delta$, without requiring its second derivative. Taking the limit in time gives the expansion of $Z_\delta$. Estimating the difference of the evolving and limiting densities gives that of $\Gamma_\delta$. Products then give the last two expansions.

\subsection{Expansion of the amplitude}
\label{subsec:amplitude-parameter-quotient}

For $0<\delta\leq\delta_0$, define
\begin{equation}\label{eq:hdelta-definition}
	h_\delta
	:=
	\frac{a_\delta-a_0}{\delta}.
\end{equation}

\begin{lemma}
	\label{lem:H1-parameter-quotient}
	One has
	\begin{equation}\label{eq:hdelta-H1-bound}
		\sup_{0<\delta\leq\delta_0}
		\norm{h_\delta(t)}_{H_v^1}
		\leq
		Ct^{C_0\varepsilon^2}(1+\log t)^3,
		\qquad
		t\geq1.
	\end{equation}
	Moreover, for every $T>1$, $h_\delta$ converges in $C([1,T];H_v^1)$ to the solution $\dot a_0$ of
	\begin{equation}\label{eq:linearized-amplitude-equation}
		\ii\partial_t\dot a_0
		+
		\frac1{4t^2}\partial_v^2\dot a_0
		=
		\frac12\kappa t^{-1}
		\left[
		\left.D_z(\abs{z}^2z)\right|_{z=a_0}\dot a_0
		+
		\abs{a_0}^2a_0
		\left(
		\log\abs{a_0}
		-
		\frac12\log(2t)
		\right)
		\right].
	\end{equation}
	All products containing $\log\abs{a_0}$ are defined by continuous extension at $a_0=0$.
\end{lemma}

\begin{proof}
	We first establish convergence of the initial quotient in $H_v^1$ using the equations for $u_\delta$ and $J(t)u_\delta$ on $[0,1]$. We then estimate its growth for $t\geq1$ and prove compact-time convergence.
	
	On $[0,1]$, use the space
	\begin{equation}\label{eq:compact-time-X-definition}
		X
		:=
		L_t^\infty L_x^2([0,1])
		\cap
		L_t^4L_x^\infty([0,1]).
	\end{equation}
	Fix
	\[
	0<\rho<\min\{1,2-\delta_0\}.
	\]
	Splitting into $0<r\leq1$ and $r\geq1$ gives, uniformly for $0\leq p\leq\delta_0$,
	\begin{align}
		r^{3+p}\abs{\log r}
		&\leq
		C_\rho
		\left(
		r^{3+p-\rho}
		+
		r^{3+p+\rho}
		\right),
		\label{eq:neighboring-power-nonlinearity}\\
		r^{2+p}(1+\abs{\log r})
		&\leq
		C_\rho
		\left(
		r^{2+p-\rho}
		+
		r^{2+p+\rho}
		\right).
		\label{eq:neighboring-power-derivative}
	\end{align}
	The left hand sides are defined to be zero at $r=0$.
	
	For $0\leq p\leq\delta_0$, let $u_p$ denote the mass-class solution of \eqref{eq:nls-delta} with exponent $p$. Since $2+p+\rho<4$, H\"older in time and \Cref{prop:uniform-preliminary-estimates} imply
	\begin{equation}\label{eq:compact-parameter-source-bound}
		\norm{\abs{u_p}^{2+p}u_p\log\abs{u_p}}_{L_t^1L_x^2([0,1])}
		\leq
		C_\rho
		\norm{u_p}_{L_t^\infty L_x^2}
		\left(
		\norm{u_p}_{L_t^4L_x^\infty}^{2+p-\rho}
		+
		\norm{u_p}_{L_t^4L_x^\infty}^{2+p+\rho}
		\right),
	\end{equation}
	where $\partial_p(\abs{z}^{2+p}z)=\abs{z}^{2+p}z\log\abs{z}$, with value zero at $z=0$.
	
	Set
	\[
	q_\delta
	:=
	\frac{u_\delta-u_0}{\delta}
	\]
	and apply the mean-value formula directly:
	\begin{equation}\label{eq:physical-parameter-quotient}
		(\ii\partial_t+\partial_x^2)q_\delta
		=\kappa\int_0^1
		\left.D_z(\abs{z}^{2+\delta}z)\right|_{z=u_0+\theta(u_\delta-u_0)}
		q_\delta\,\dd\theta+\kappa\int_0^1\abs{u_0}^{2+\theta\delta}u_0\log\abs{u_0}\,\dd\theta.
	\end{equation}
	The same H\"older estimate used in \eqref{eq:mass-subcritical-nonlinearity}, together with \eqref{eq:compact-parameter-source-bound}, yields
	\[
	\begin{aligned}
		\norm{\int_0^1
			\left.D_z(\abs{z}^{2+\delta}z)\right|_{z=u_0+\theta(u_\delta-u_0)}
			q_\delta\,\dd\theta}_{L_t^1L_x^2}
		&\leq C\varepsilon^2\norm{q_\delta}_X,\\
		\norm{\int_0^1\abs{u_0}^{2+\theta\delta}u_0\log\abs{u_0}\,\dd\theta}_{L_t^1L_x^2}
		&\leq C_\rho\varepsilon^{3-\rho}.
	\end{aligned}
	\]
	Endpoint Strichartz and absorption therefore give
	\begin{equation}\label{eq:compact-q-bound}
		\sup_{0<\delta\leq\delta_0}
		\norm{q_\delta}_X
		\leq
		C_\rho\varepsilon^{3-\rho}.
	\end{equation}
	
	Let $\dot u_0\in X$ solve
	\begin{equation}\label{eq:physical-linearized-parameter-equation}
		(\ii\partial_t+\partial_x^2)\dot u_0
		=
		\kappa
		\left(
		\left.D_z(\abs{z}^2z)\right|_{z=u_0}\dot u_0
		+
		\abs{u_0}^2u_0\log\abs{u_0}
		\right),
		\qquad
		\dot u_0(0)=0.
	\end{equation}
	The logarithmic term is defined by continuous extension at zero.
	Since
	\[
	u_\delta-u_0
	=
	\delta q_\delta
	\longrightarrow0
	\qquad
	\text{in }X,
	\]
	the state coefficients in \eqref{eq:physical-parameter-quotient} converge to $\left.D_z(\abs{z}^2z)\right|_{z=u_0}$. The parameter sources converge by the mean-value formula, pointwise convergence, and the integrable majorant \eqref{eq:neighboring-power-nonlinearity}. The endpoint Strichartz stability estimate, with the same small-coefficient absorption, gives
	\begin{equation}\label{eq:compact-q-convergence}
		q_\delta
		\longrightarrow
		\dot u_0
		\qquad
		\text{in }X.
	\end{equation}
	
	We next differentiate covariantly. For $0\leq p\leq\delta_0$, put
	\[
	V_p
	:=
	J(t)u_p.
	\]
	The identity \eqref{eq:J-nonlinearity-identity} gives
	\begin{equation}\label{eq:Vp-equation}
		(\ii\partial_t+\partial_x^2)V_p
		=
		\kappa G_p(u_p,V_p),
	\end{equation}
	where
	\begin{equation}\label{eq:Gp-definition}
		G_p(u,V)
		:=
		\left(2+\frac p2\right)\abs{u}^{2+p}V
		-
		\left(1+\frac p2\right)\abs{u}^pu^2\overline V.
	\end{equation}
	All $V_p$ have initial value $xu_{\mathrm{in}}$, and \eqref{eq:finite-time-J-bound} gives
	\[
	\sup_{0\leq p\leq\delta_0}
	\norm{V_p}_X
	\leq
	C\varepsilon.
	\]
	The quotient satisfies
	\[
	\frac{V_\delta-V_0}{\delta}=J(t)q_\delta.
	\]
	A mean value first in the state and then in the parameter yields
	\begin{align}
		\frac{
			G_\delta(u_\delta,V_\delta)-G_0(u_0,V_0)
		}{\delta}
		=&
		\int_0^1
		D_{(u,V)}G_\delta
		\bigl(
		u_0+\theta(u_\delta-u_0),
		V_0+\theta(V_\delta-V_0)
		\bigr)
		[q_\delta,\frac{V_\delta-V_0}{\delta}]
		\,\dd\theta
		\notag\\
		&+
		\int_0^1
		\partial_pG_p(u_0,V_0)\big|_{p=\theta\delta}
		\,\dd\theta.
		\label{eq:G-parameter-decomposition}
	\end{align}
	Uniformly for $p\in[0,\delta_0]$, real differentiation gives
	\begin{align*}
		\abs{D_VG_p(u,V)[R]}
		&\leq
		C\abs{u}^{2+p}\abs R,\\
		\abs{D_uG_p(u,V)[Q]}
		&\leq
		C\abs{u}^{1+p}\abs Q\abs V,\\
		\abs{\partial_pG_p(u,V)}
		&\leq
		C\abs{u}^{2+p}
		\bigl(
		1+\abs{\log\abs u}
		\bigr)
		\abs V.
	\end{align*}
	The expressions are extended continuously at $u=0$. In particular, the powers of $u$ cancel the apparent singular factors arising from differentiation of $\abs u^p$.
	
	Using \eqref{eq:neighboring-power-derivative}, with $\frac{V_\delta-V_0}{\delta},V_p$ in $L_t^\infty L_x^2$ and $u_p,q_\delta$ in $L_t^4L_x^\infty$, endpoint Strichartz gives
	\begin{equation}\label{eq:compact-Jq-bound}
		\norm{\frac{V_\delta-V_0}{\delta}}_X
		\leq
		C\varepsilon^2\norm{\frac{V_\delta-V_0}{\delta}}_X
		+
		C\varepsilon^{1-\rho}
		\norm{q_\delta}_X
		\sup_p\norm{V_p}_X
		+
		C_\rho
		\left(
		\varepsilon^{2-\rho}
		+
		\varepsilon^{2+\delta_0+\rho}
		\right)
		\norm{V_0}_X.
	\end{equation}
	After absorption this bounds $\frac{V_\delta-V_0}{\delta}$ uniformly in $X$.
	The equation for the limiting covariant component is obtained from
	\eqref{eq:G-parameter-decomposition} by replacing
	\[
	(u_\delta,V_\delta,q_\delta,\frac{V_\delta-V_0}{\delta},\delta)
	\]
	with the limiting state and $\dot u_0$. Its covariant component is
	identified distributionally with $J(t)\dot u_0$, as in the weight
	argument above.
	State convergence follows from
	\eqref{eq:compact-q-convergence} and
	\[
	V_\delta-V_0
	\longrightarrow0
	\qquad
	\text{in }X.
	\]
	The explicit source converges by
	\eqref{eq:neighboring-power-derivative}. Strichartz stability and
	absorption then give
	\begin{equation}\label{eq:compact-Jq-convergence}
		\frac{V_\delta-V_0}{\delta}
		\longrightarrow
		J(t)\dot u_0
		\qquad
		\text{in }X
	\end{equation}
	in the distributional sense.
	
	All four functions in \eqref{eq:compact-q-convergence}--\eqref{eq:compact-Jq-convergence} have continuous $L_x^2$ representatives because their Duhamel sources lie in $L_t^1L_x^2$. Hence the convergence may be evaluated at $t=1$. The self-similar transformation gives
	\[
	h_\delta(1,v)
	=
	2^{1/2}\ee^{-\ii v^2}q_\delta(1,2v).
	\]
	Direct differentiation, using $J(1)=x+2\ii\partial_x$, gives
	\[
	\partial_vh_\delta(1,v)
	=
	-\ii2^{1/2}\ee^{-\ii v^2}
	\bigl(
	J(1)q_\delta
	\bigr)(1,2v).
	\]
	Therefore
	\begin{equation}\label{eq:initial-profile-quotient-convergence}
		h_\delta(1)
		\longrightarrow
		\dot a_0(1)
		\qquad
		\text{in }H_v^1,
	\end{equation}
	where
	\[
	\dot a_0(1,v)
	=
	2^{1/2}\ee^{-\ii v^2}\dot u_0(1,2v).
	\]
	The quotients are uniformly bounded in $H_v^1$. In particular, the derivative of $h_\delta(1)$ comes from $J(1)q_\delta$, not from an assumed physical derivative of $q_\delta$.
	
	We now estimate the quotient for $t\geq1$. The amplitude equation \eqref{eq:exact-amplitude-equation} and the mean-value formula give
	\begin{equation}\label{eq:exact-hdelta-equation}
		\begin{aligned}
			\ii\partial_th_\delta+\frac1{4t^2}\partial_v^2h_\delta
			=&\kappa 2^{-1-\delta/2}t^{-1-\delta/2}
			\int_0^1\left.D_z(\abs{z}^{2+\delta}z)
			\right|_{z=a_0+\theta(a_\delta-a_0)}h_\delta\,\dd\theta\\
			&+\kappa t^{-1}
			\frac{2^{-1-\delta/2}t^{-\delta/2}\abs{a_0}^{2+\delta}a_0
				-\frac12\abs{a_0}^2a_0}{\delta}.
		\end{aligned}
	\end{equation}
	
	The identity $2^{-1-\delta/2}t^{-\delta/2}=\frac12\exp(-\frac\delta2\log(2t))$ gives
	\begin{equation}\label{eq:time-coefficient-quotient-bound}
		\abs{
			\frac{2^{-1-\delta/2}t^{-\delta/2}-2^{-1}}{\delta}
		}
		\leq
		C(1+\log t).
	\end{equation}
	The exponent derivative obeys, for $0\leq p\leq\delta_0$,
	\begin{equation}\label{eq:nonlinearity-parameter-H1-pointwise}
		\abs{\abs{z}^{2+p}z\log\abs{z}}\leq C\abs{z}^{5/2},
		\qquad
		\abs{D_z(\abs{z}^{2+p}z\log\abs{z})}\leq C\abs{z}^{3/2},
	\end{equation}
	uniformly for $\abs z\leq1$. Indeed, $r^{1/2}(1+\abs{\log r})$ is bounded on $(0,1]$, so both expressions extend continuously at $z=0$. The mean-value formula, \eqref{eq:a-H1-coupled-improvement}, and \eqref{eq:uniform-amplitude-Linfty} give
	\begin{equation}\label{eq:parameter-source-H1}
		\norm{\frac{2^{-1-\delta/2}t^{-\delta/2}\abs{a_0(t)}^{2+\delta}a_0(t)-\frac12\abs{a_0(t)}^2a_0(t)}{\delta}}_{H^1}
		\leq
		C\varepsilon^{5/2}
		t^{C\varepsilon^2}
		(1+\log t).
	\end{equation}
	This estimate uses no second spatial derivative of $\abs{z}^{2+p}z\log\abs{z}$.
	
	By \Cref{lem:uniform-tame-estimate},
	\begin{align}
		\norm{\int_0^1
			\left.D_z(\abs{z}^{2+\delta}z)\right|_{z=a_0+\theta(a_\delta-a_0)}\,\dd\theta}_{L^\infty}
		&\leq C\varepsilon^2,
		\label{eq:linearized-coefficient-Linfty}\\
		\norm{\int_0^1
			\left.D_z^2(\abs{z}^{2+\delta}z)\right|_{z=a_0+\theta(a_\delta-a_0)}
			[\partial_v\bigl(a_0+\theta(a_\delta-a_0)\bigr)]\,\dd\theta}_{L^2}
		&\leq C\varepsilon^2t^{C\varepsilon^2}.
		\label{eq:linearized-coefficient-derivative}
	\end{align}
	The second integral is the $v$ derivative of the first and uses only second real derivatives of the nonlinearity.
	
	Since the linear Schr\"odinger term is skew-adjoint, \eqref{eq:linearized-coefficient-Linfty} and \eqref{eq:parameter-source-H1} give
	\begin{equation}\label{eq:H1-quotient-D0}
		\frac{\dd}{\dd t}
		\norm{h_\delta(t)}_{L^2}
		\leq
		C\varepsilon^2t^{-1}
		\norm{h_\delta(t)}_{L^2}
		+
		Ct^{-1+C\varepsilon^2}
		(1+\log t).
	\end{equation}
	Gr\"onwall's inequality gives
	\begin{equation}\label{eq:H1-quotient-D0-bound}
		\norm{h_\delta(t)}_{L^2}
		\leq
		Ct^{C\varepsilon^2}
		(1+\log t)^2.
	\end{equation}
	
	At order one, the product rule gives
	\[
	\begin{aligned}
	&\partial_v\int_0^1
	\left.D_z(\abs{z}^{2+\delta}z)\right|_{z=a_0+\theta(a_\delta-a_0)}
	h_\delta\,\dd\theta\\
	&\quad=\int_0^1
	\left.D_z(\abs{z}^{2+\delta}z)\right|_{z=a_0+\theta(a_\delta-a_0)}
	\partial_vh_\delta\,\dd\theta\\
	&\qquad+\int_0^1
	\left.D_z^2(\abs{z}^{2+\delta}z)\right|_{z=a_0+\theta(a_\delta-a_0)}
	[\partial_v\bigl(a_0+\theta(a_\delta-a_0)\bigr),h_\delta]\,\dd\theta.
	\end{aligned}
	\]
	Using
	\[
	\norm{h_\delta}_{L^\infty}
	\leq
	C
	\norm{h_\delta}_{L^2}^{1/2}
	\norm{h_\delta}_{H^1}^{1/2},
	\]
	we obtain
	\begin{equation}
		\frac{\dd}{\dd t}
		\norm{\partial_vh_\delta(t)}_{L^2}
		\leq
		C\varepsilon^2t^{-1}
		\norm{\partial_vh_\delta(t)}_{L^2}
		+
		C\varepsilon^2
		t^{-1+C\varepsilon^2}
		\norm{h_\delta}_{L^2}^{1/2}
		\norm{h_\delta}_{H^1}^{1/2}
		+
		Ct^{-1+C\varepsilon^2}
		(1+\log t).
		\label{eq:H1-quotient-D1}
	\end{equation}
	Using $\norm{h_\delta}_{H^1}\leq
	\norm{h_\delta}_{L^2}+\norm{\partial_vh_\delta}_{L^2}$ and Young's
	inequality gives
	\begin{equation}\label{eq:H1-quotient-D1-young}
		\frac{\dd}{\dd t}
		\norm{\partial_vh_\delta(t)}_{L^2}
		\leq
		C\varepsilon^2t^{-1}
		\norm{\partial_vh_\delta(t)}_{L^2}
		+
		C\varepsilon^2
		t^{-1+C\varepsilon^2}
		\norm{h_\delta(t)}_{L^2}
		+
		Ct^{-1+C\varepsilon^2}
		(1+\log t).
	\end{equation}
	Young's inequality doubles the polynomial exponent constant in the mixed term; this is included in the new absolute $C$ in \eqref{eq:H1-quotient-D1-young}. Substituting \eqref{eq:H1-quotient-D0-bound} adds its exponent constant, so both forcing terms are bounded by
\[
Ct^{-1+C\varepsilon^2}(1+\log t)^2.
\]
	Gr\"onwall's inequality therefore gives, with a possibly larger absolute exponent constant,
\[
\norm{\partial_vh_\delta(t)}_{L^2}
\leq Ct^{C\varepsilon^2}
\left(1+\int_1^t s^{-1}(1+\log s)^2\,\dd s\right)
\leq Ct^{C\varepsilon^2}(1+\log t)^3.
\]
	Here the Gr\"onwall factor and the forcing power are both bounded by powers of $t$ before integrating $s^{-1}(1+\log s)^2$; no negative power of $\varepsilon$ is introduced. Together with \eqref{eq:H1-quotient-D0-bound}, this proves the actual bound $\norm{h_\delta(t)}_{H^1}\leq Ct^{C\varepsilon^2}(1+\log t)^3$. Absorbing these exponent constants into the single final $C_0$ proves \eqref{eq:hdelta-H1-bound}.
	
	We finally prove the compact-time convergence of $h_\delta$. On every fixed interval $[1,T]$, \eqref{eq:hdelta-H1-bound} gives
	\[
	a_\delta-a_0
	=
	\delta h_\delta
	\longrightarrow0
	\qquad
	\text{in }H_v^1,
	\]
	and hence also in $L_v^\infty$. The state coefficients converge by the mean-value formula. For the differentiated coefficient, write
	\begin{align*}
		&\left.D_z^2(\abs{z}^{2+\delta}z)\right|_{z=a_0+\theta(a_\delta-a_0)}\partial_v\bigl(a_0+\theta(a_\delta-a_0)\bigr)
		-
		\left.D_z^2(\abs{z}^2z)\right|_{z=a_0}\partial_va_0\\
		&\quad=
		\left.D_z^2(\abs{z}^{2+\delta}z)\right|_{z=a_0+\theta(a_\delta-a_0)}
		\bigl(
		\partial_v\bigl(a_0+\theta(a_\delta-a_0)\bigr)-\partial_va_0
		\bigr)\\
		&\qquad+
		\bigl(
		\left.D_z^2(\abs{z}^{2+\delta}z)\right|_{z=a_0+\theta(a_\delta-a_0)}-\left.D_z^2(\abs{z}^2z)\right|_{z=a_0}
		\bigr)
		\partial_va_0.
	\end{align*}
	The first term tends to zero in $L_v^2$, since $a_0+\theta(a_\delta-a_0)\to a_0$ strongly in $H_v^1$ and $D_z^2(\abs{z}^{2+\delta}z)$ is uniformly bounded on the common bounded range. Moreover, $D_z^2(\abs{z}^{2+\delta}z)\to D_z^2(\abs{z}^2z)$ uniformly on bounded subsets of $\C$, and $a_0+\theta(a_\delta-a_0)\to a_0$ in $L_v^\infty$. Hence the coefficient in the second term tends to zero in $L_v^\infty$, and the second term tends to zero in $L_v^2$. The source converges in $H_v^1$ by \eqref{eq:nonlinearity-parameter-H1-pointwise}. The $H_v^1$ stability estimate for the difference now proves the compact-time convergence and \eqref{eq:linearized-amplitude-equation}.
\end{proof}

\subsection{Expansion after conjugation}
\label{subsec:conjugation-remainder-quotient}

Define
\begin{equation}\label{eq:kdelta-definition}
	k_\delta
	:=
	\frac{A_\delta-A_0}{\delta}
	=
	S(t)h_\delta,
\end{equation}
and recall from \eqref{eq:conjugation-remainder-definition} that
\begin{equation}\label{eq:Rsharp-along-solution}
	\mathcal R_\delta^\sharp(t)
	=
	\mathcal R_\delta(t,A_\delta(t)).
\end{equation}
By the unitarity of $S(t)$ on $H^1$ and the actual polynomial bound proved in \Cref{lem:H1-parameter-quotient},
\begin{equation}\label{eq:kdelta-H1-bound}
	\norm{k_\delta(t)}_{H^1}
	\leq
	Ct^{C\varepsilon^2}(1+\log t)^3.
\end{equation}

\begin{lemma}
	\label{lem:conjugation-remainder-quotient}
	Uniformly for $0<\delta\leq\delta_0$,
	\begin{equation}\label{eq:Rsharp-quotient-bound}
		\norm{
			\frac{\mathcal R_\delta^\sharp(t)-\mathcal R_0^\sharp(t)}{\delta}
		}_{L^2}
		\leq
		Ct^{-1/2+C_0\varepsilon^2}(1+\log t)^3.
	\end{equation}
\end{lemma}

\begin{proof}
	For $0\leq p\leq\delta_0$, write
	\[
	\mathcal R_p(t,A)
	=
	S\bigl(\abs{S^{-1}A}^{2+p}S^{-1}A\bigr)-\abs{A}^{2+p}A,
	\]
	where $S=S(t)$. The decomposition is
	\begin{equation}\label{eq:Rsharp-state-parameter-decomposition}
		\frac{\mathcal R_\delta^\sharp-\mathcal R_0^\sharp}{\delta}
		=
		\frac{\mathcal R_\delta(A_\delta)-\mathcal R_\delta(A_0)}{\delta}
		+
		\frac{\mathcal R_\delta(A_0)-\mathcal R_0(A_0)}{\delta}.
	\end{equation}
	
	For the state variation, the mean-value formula gives
	\begin{equation}\label{eq:conjugation-remainder-state-mean-value}
		\frac{\mathcal R_\delta(A_\delta)-\mathcal R_\delta(A_0)}{\delta}
		=\int_0^1D_A\mathcal R_\delta
		\bigl(A_0+\theta(A_\delta-A_0)\bigr)[k_\delta]\,\dd\theta.
	\end{equation}
	For fixed $p$,
	\[
	D_A\mathcal R_p(A)[K]
	=
	S\bigl(\left.D_z(\abs{z}^{2+p}z)\right|_{z=a}k\bigr)-\left.D_z(\abs{z}^{2+p}z)\right|_{z=A}K,
	\]
	where $a=S^{-1}A$ and $k=S^{-1}K$. We use
	\begin{equation}\label{eq:conjugation-remainder-state-three-term-decomposition}
		\begin{aligned}
			D_A\mathcal R_p(A)[K]
			=&(S-I)\left[\left.D_z(\abs{z}^{2+p}z)\right|_{z=a}k\right]\\
			&+\left[\left.D_z(\abs{z}^{2+p}z)\right|_{z=a}
			-\left.D_z(\abs{z}^{2+p}z)\right|_{z=A}\right]k\\
			&+\left.D_z(\abs{z}^{2+p}z)\right|_{z=A}(k-K).
		\end{aligned}
	\end{equation}
	
	For the first term,
	\[
	\partial_v(\left.D_z(\abs{z}^{2+p}z)\right|_{z=a}k)
	=
	\left.D_z^2(\abs{z}^{2+p}z)\right|_{z=a}[\partial_va,k]
	+
	\left.D_z(\abs{z}^{2+p}z)\right|_{z=a}\partial_vk.
	\]
	Using $H^1\hookrightarrow L^\infty$, \eqref{eq:a-H1-coupled-improvement}, and \eqref{eq:kdelta-H1-bound}, we obtain
	\[
	\norm{\left.D_z(\abs{z}^{2+p}z)\right|_{z=a}k}_{H^1}
	\leq
	C\varepsilon^2
	t^{C\varepsilon^2}(1+\log t)^3.
	\]
	Hence \eqref{eq:small-propagator-L2} gives
	\begin{equation}\label{eq:conjugation-remainder-state-term-one}
		\norm{(S-I)[\left.D_z(\abs{z}^{2+p}z)\right|_{z=a}k]}_{L^2}
		\leq
		C\varepsilon^2
		t^{-1/2+C\varepsilon^2}(1+\log t)^3.
	\end{equation}
	
	For the middle term, we retain the full $t^{-1/2}$ gain by placing the coefficient difference in $L^2$, not in $L^\infty$. By \eqref{eq:uniform-nonlinearity-derivative-Lipschitz},
	\begin{equation}\label{eq:conjugation-remainder-coefficient-difference-L2}
		\norm{\left.D_z(\abs{z}^{2+p}z)\right|_{z=a}-\left.D_z(\abs{z}^{2+p}z)\right|_{z=A}}_{L^2}
		\leq
		C\varepsilon^{1+p}\norm{a-A}_{L^2}
		\leq
		C\varepsilon^{2+p}
		t^{-1/2+C\varepsilon^2}.
	\end{equation}
	Since
	\[
	\norm{k}_{L^\infty}
	\leq
	C\norm{K}_{H^1},
	\]
	we obtain
	\begin{equation}\label{eq:conjugation-remainder-state-term-two}
		\norm{[\left.D_z(\abs{z}^{2+p}z)\right|_{z=a}-\left.D_z(\abs{z}^{2+p}z)\right|_{z=A}]k}_{L^2}
		\leq
		C\varepsilon^2
		t^{-1/2+C\varepsilon^2}(1+\log t)^3.
	\end{equation}
	
	Finally,
	\[
	\norm{k-K}_{L^2}
	\leq
	Ct^{-1/2}\norm{K}_{H^1},
	\]
	so
	\begin{equation}\label{eq:conjugation-remainder-state-term-three}
		\norm{\left.D_z(\abs{z}^{2+p}z)\right|_{z=A}(k-K)}_{L^2}
		\leq
		C\varepsilon^2
		t^{-1/2+C\varepsilon^2}(1+\log t)^3.
	\end{equation}
	Equations \eqref{eq:conjugation-remainder-state-term-one}-- \eqref{eq:conjugation-remainder-state-term-three} control the state variation in \eqref{eq:Rsharp-state-parameter-decomposition}.
	
	For the explicit parameter variation, differentiate the nonlinearity directly.
	Then
	\begin{equation}\label{eq:conjugation-remainder-explicit-parameter-derivative}
		\partial_p\mathcal R_p(A)
		=
		S\bigl(\abs{S^{-1}A}^{2+p}S^{-1}A\log\abs{S^{-1}A}\bigr)-\abs{A}^{2+p}A\log\abs{A}.
	\end{equation}
	By \eqref{eq:nonlinearity-parameter-H1-pointwise}, both $\abs{z}^{2+p}z\log\abs{z}$ and $D_z(\abs{z}^{2+p}z\log\abs{z})$ extend continuously at zero.
	The mean-value formula gives
	\[
	\frac{\mathcal R_\delta(A_0)-\mathcal R_0(A_0)}{\delta}
	=
	\int_0^1
	\partial_p\mathcal R_p(A_0)\big|_{p=\theta\delta}
	\,\dd\theta.
	\]
	Writing $a_0=S^{-1}A_0$, decompose
	\[
	\partial_p\mathcal R_p(A_0)
	=
	(S-I)\bigl(\abs{a_0}^{2+p}a_0\log\abs{a_0}\bigr)
	+
	\abs{a_0}^{2+p}a_0\log\abs{a_0}-\abs{A_0}^{2+p}A_0\log\abs{A_0}.
	\]
	Equation \eqref{eq:nonlinearity-parameter-H1-pointwise} gives
	\[
	\norm{\abs{a_0}^{2+p}a_0\log\abs{a_0}}_{H^1}
	\leq
	C\varepsilon^{5/2}
	t^{C\varepsilon^2}.
	\]
	Hence
	\begin{align}
		\norm{(S-I)\bigl(\abs{a_0}^{2+p}a_0\log\abs{a_0}\bigr)}_{L^2}
		&\leq
		C\varepsilon^{5/2}
		t^{-1/2+C\varepsilon^2},
		\label{eq:conjugation-remainder-parameter-term-one}\\
		\norm{\abs{a_0}^{2+p}a_0\log\abs{a_0}-\abs{A_0}^{2+p}A_0\log\abs{A_0}}_{L^2}
		&\leq
		C\varepsilon^{3/2}\norm{a_0-A_0}_{L^2}
		\leq
		C\varepsilon^{5/2}
		t^{-1/2+C\varepsilon^2}.
		\label{eq:conjugation-remainder-parameter-term-two}
	\end{align}
	The products above add the absolute exponent constants for the amplitude and $k_\delta$; each term retains the $t^{-1/2}$ gain. Combining the state and explicit parameter estimates gives
\[
\norm{\frac{\mathcal R_\delta^\sharp(t)-\mathcal R_0^\sharp(t)}{\delta}}_{L^2}
\leq Ct^{-1/2+C\varepsilon^2}(1+\log t)^3.
\]
	Absorbing this exponent constant into $C_0$ proves \eqref{eq:Rsharp-quotient-bound}. Only $H^1$ control of $k_\delta$ has been used.
\end{proof}

\subsection{Expansion of \texorpdfstring{$Z_\delta$}{Z}}
\label{subsec:Z-expansion}

We first differentiate the evolving phase, then pass to the limit in $b_\delta$.

\begin{lemma}
	\label{lem:phase-quotient-H1}
	Uniformly for $0<\delta\leq\delta_0$,
	\begin{equation}\label{eq:phase-quotient-H1-bound}
		\norm{\frac{\Theta_\delta(t)-\Theta_0(t)}{\delta}}_{H^1}
		\leq
		Ct^{C_0\varepsilon^2}(1+\log t)^4.
	\end{equation}
	For every fixed $t$, $\frac{\Theta_\delta(t)-\Theta_0(t)}{\delta}$ converges in $H^1$ to
	\begin{equation}\label{eq:Theta-dot-formula}
		\dot\Theta_0(t)
		=
		\frac{\kappa}{2}
		\int_1^t
		s^{-1}
		\left[
		2\operatorname{Re}(\overline{A_0}\dot A_0)
		+
		\abs{A_0}^2
		\left(
		\log\abs{A_0}
		-
		\frac12\log(2s)
		\right)
		\right]
		\,\dd s,
	\end{equation}
	where
	\[
	\dot A_0(s):=S(s)\dot a_0(s).
	\]
\end{lemma}

\begin{proof}
	The phase formula gives
	\begin{equation}\label{eq:phase-quotient-integral-formula}
		\begin{aligned}
			\frac{\Theta_\delta(t)-\Theta_0(t)}{\delta}
			=&\kappa\int_1^t2^{-1-\delta/2}s^{-1-\delta/2}
			\frac{\abs{A_\delta(s)}^{2+\delta}-\abs{A_0(s)}^{2+\delta}}{\delta}\,\dd s\\
			&+\kappa\int_1^ts^{-1}
			\frac{2^{-1-\delta/2}s^{-\delta/2}\abs{A_0(s)}^{2+\delta}
				-\frac12\abs{A_0(s)}^2}{\delta}\,\dd s.
		\end{aligned}
	\end{equation}
	On the bounded range of the amplitudes,
	\[
	\abs{D_z(\abs{z}^{2+p})}
	\leq
	C\abs{z}^{1+p},
	\qquad
	\abs{D_z^2(\abs{z}^{2+p})}
	\leq
	C\abs{z}^{p}
	\leq C.
	\]
	The state quotient in the first integral of \eqref{eq:phase-quotient-integral-formula} is an average of $\left.D_z(\abs{z}^{2+\delta})\right|_{z=A_0+\theta(A_\delta-A_0)}k_\delta$. Using \eqref{eq:a-H1-coupled-improvement} and \eqref{eq:kdelta-H1-bound}, the differentiated product adds their absolute exponent constants. Hence
	\begin{equation}\label{eq:phase-state-quotient-H1}
		\norm{\left.D_z(\abs{z}^{2+\delta})\right|_{z=A_0+\theta(A_\delta-A_0)}k_\delta}_{H^1}
		\leq
		C\varepsilon
		t^{C\varepsilon^2}(1+\log t)^3.
	\end{equation}
	
	The explicit derivative
	\[
	\partial_p(\abs{z}^{2+p})
	=
	\abs{z}^{2+p}\log\abs{z}
	\]
	and its first real derivative extend continuously at zero, since
	\[
	\abs{D_z(\abs{z}^{2+p}\log\abs{z})}
	\leq
	C\abs{z}^{1+p}
	\left(
	1+\abs{\log\abs z}
	\right).
	\]
	Together with \eqref{eq:time-coefficient-quotient-bound}, this gives
	\begin{equation}\label{eq:phase-explicit-quotient-H1}
		\norm{
			\frac{2^{-1-\delta/2}t^{-\delta/2}\abs{A_0(t)}^{2+\delta}
				-\frac12\abs{A_0(t)}^2}{\delta}}_{H^1}
		\leq Ct^{C\varepsilon^2}(1+\log t).
	\end{equation}
	Integrating \eqref{eq:phase-state-quotient-H1} and \eqref{eq:phase-explicit-quotient-H1} against $s^{-1}$ gives
\[
\norm{\frac{\Theta_\delta(t)-\Theta_0(t)}{\delta}}_{H^1}
\leq Ct^{C\varepsilon^2}\int_1^t s^{-1}(1+\log s)^3\,\dd s
\leq Ct^{C\varepsilon^2}(1+\log t)^4.
\]
	Absorbing this absolute exponent constant into $C_0$ proves \eqref{eq:phase-quotient-H1-bound}.
	
	We now prove the fixed-time $H^1$ convergence, including the zero set of $A_0$. On every fixed interval $[1,T]$, \Cref{lem:H1-parameter-quotient} gives
	\begin{equation}\label{eq:kdelta-compact-H1-convergence}
		k_\delta
		\longrightarrow
		\dot A_0
		\quad\text{in }C([1,T];H^1),
		\qquad
		A_\delta
		\longrightarrow
		A_0
		\quad\text{in }C([1,T];H^1).
	\end{equation}
	For the state quotient, the mean-value formula gives
	\[
	\frac{
		\abs{A_\delta}^{2+\delta}
		-
		\abs{A_0}^{2+\delta}
	}{\delta}
	=
	\int_0^1
	\left.D_z(\abs{z}^{2+\delta})\right|_{z=A_0+\theta\delta k_\delta}k_\delta
	\,\dd\theta.
	\]
	Since $D_z(\abs{z}^{2+\delta})\to D_z(\abs{z}^2)$ uniformly on bounded subsets of $\C$, \eqref{eq:kdelta-compact-H1-convergence} gives
	\[
	\frac{\abs{A_\delta}^{2+\delta}-\abs{A_0}^{2+\delta}}{\delta}
	\longrightarrow
	\left.D_z(\abs{z}^2)\right|_{z=A_0}\dot A_0
	\qquad
	\text{in }L^2.
	\]
	
	The same argument treats the term containing
	$\left.D_z(\abs{z}^{2+\delta})\right|_{z=A_0+\theta\delta k_\delta}\partial_vk_\delta$ in
	\[
	\begin{aligned}
	&\partial_v\left(\frac{\abs{A_\delta}^{2+\delta}-\abs{A_0}^{2+\delta}}{\delta}\right)\\
	&\quad=
	\int_0^1
	\left.D_z^2(\abs{z}^{2+\delta})\right|_{z=A_0+\theta\delta k_\delta}
	[\partial_v(A_0+\theta\delta k_\delta),k_\delta]\,\dd\theta\\
	&\qquad+\int_0^1
	\left.D_z(\abs{z}^{2+\delta})\right|_{z=A_0+\theta\delta k_\delta}\partial_vk_\delta
	\,\dd\theta.
	\end{aligned}
	\]
	It remains to treat the Hessian term. For $z\neq0$, its
	real Hessian is
	\begin{equation}\label{eq:density-exact-real-Hessian}
		D_z^2(\abs{z}^{2+p})[h,k]
		=
		(2+p)\abs{z}^{p}\operatorname{Re}(\overline h k)
		+
		p(2+p)\abs{z}^{p-2}
		\operatorname{Re}(\overline z h)
		\operatorname{Re}(\overline z k).
	\end{equation}
	The second term in \eqref{eq:density-exact-real-Hessian}, evaluated at $p=\delta$, has $L^2$ norm at most
	\[
	C\delta
	\norm{\partial_v(A_0+\theta\delta k_\delta)}_{L^2}
	\norm{k_\delta}_{L^\infty},
	\]
	and therefore tends to zero. This estimate remains valid when $A_0+\theta\delta k_\delta=0$ by defining that term to be zero.
	
	It remains to pass to the limit in the first term of \eqref{eq:density-exact-real-Hessian}. The differences containing $\partial_v(A_0+\theta\delta k_\delta)-\partial_vA_0$ or $k_\delta-\dot A_0$ tend to zero in $L^2$, by \eqref{eq:kdelta-compact-H1-convergence} and
	$H^1\hookrightarrow L^\infty$. After those differences and the factor $2+\delta$ are removed, the remaining expression is bounded by a constant times
	\begin{equation}\label{eq:zero-set-multiplier-term}
		\left(
		\abs{A_0+\theta\delta k_\delta}^\delta-1
		\right)
		\partial_vA_0.
	\end{equation}
	
	Fix $s\in[1,T]$ and a level $\rho>0$. On $\{\abs{A_0(s)}>\rho\}$, the uniform convergence in \eqref{eq:kdelta-compact-H1-convergence} gives
	\[
	\abs{A_0+\theta\delta k_\delta}
	\geq
	\frac{\rho}{2},
	\]
	uniformly in $0\leq\theta\leq1$, once $\delta$ is small. Hence
	\[
	\sup_{\{\abs{A_0(s)}>\rho\},\,0\leq\theta\leq1}
	\abs{
		\abs{A_0+\theta\delta k_\delta}^\delta-1
	}
	\longrightarrow0.
	\]
	On $\{\abs{A_0(s)}\leq\rho\} $, after decreasing $\varepsilon_0$ so that the common $L^\infty$ bound is at most one, the multiplier in \eqref{eq:zero-set-multiplier-term} has absolute value at most one.
	Moreover, an $H^1(\R)$ function has a locally absolutely continuous representative, and its derivative vanishes almost everywhere on each level set. Thus
	\[
	\partial_vA_0(s)=0
	\qquad
	\text{a.e. on }\{A_0(s)=0\},
	\]
	and dominated convergence gives
	\[
	\norm{
		\partial_vA_0(s)
		\mathbf 1_{\{\abs{A_0(s)}\leq\rho\}}
	}_{L^2}
	\longrightarrow0
	\qquad
	(\rho\downarrow0).
	\]
	Taking first $\delta\to0$ with $\rho$ fixed and then $\rho\downarrow0$ proves that \eqref{eq:zero-set-multiplier-term} tends to zero in $L^2$, uniformly in $\theta$. Consequently, for every fixed $s$,
	\[
	\frac{\abs{A_\delta(s)}^{2+\delta}-\abs{A_0(s)}^{2+\delta}}{\delta}
	\longrightarrow
	\left.D_z(\abs{z}^2)\right|_{z=A_0(s)}\dot A_0(s)
	=
	2\operatorname{Re}
	\left(
	\overline{A_0(s)}\dot A_0(s)
	\right)
	\qquad
	\text{in }H^1.
	\]
	This argument avoids uniform convergence of $D_z^2(\abs{z}^{2+\delta})$ at $z=0$, which fails.
	
	For the explicit exponent quotient,
	\[
	\partial_p(\abs{z}^{2+p})
	=
	\abs{z}^{2+p}\log\abs{z},
	\qquad
	\abs{D_z(\abs{z}^{2+p}\log\abs{z})}
	\leq
	C\abs{z}^{1+p}
	\left(
	1+\abs{\log\abs z}
	\right).
	\]
	The mean-value formula and dominated convergence give
	\[
	\frac{
		2^{-1-\delta/2}s^{-\delta/2}\abs{A_0(s)}^{2+\delta}
		-
		2^{-1}\abs{A_0(s)}^{2}
	}{\delta}
	\longrightarrow
	\frac12\abs{A_0(s)}^2
	\left(
	\log\abs{A_0(s)}
	-
	\frac12\log(2s)
	\right)
	\qquad
	\text{in }H^1.
	\]
	The preceding two convergences are dominated on $[1,t]$ by \eqref{eq:phase-state-quotient-H1} and \eqref{eq:phase-explicit-quotient-H1}. Bochner dominated convergence in \eqref{eq:phase-quotient-integral-formula} therefore gives
	\[
	\frac{\Theta_\delta(t)-\Theta_0(t)}{\delta}
	\longrightarrow
	\dot\Theta_0(t)
	\qquad
	\text{in }H^1
	\]
	for every fixed $t$, with $\dot\Theta_0$ given by \eqref{eq:Theta-dot-formula}.
\end{proof}

\begin{proposition}
	\label{prop:Z-first-order-expansion}
	There exists $\dot Z_0\in L^2(\R)$ such that
	\begin{equation}\label{eq:Z-first-order-section-five}
		\frac{Z_\delta-Z_0}{\delta}
		\longrightarrow
		\dot Z_0
		\qquad
		\text{in }L^2.
	\end{equation}
	More precisely, define
	\begin{align}
		\dot{\mathcal R}_0^\sharp(t)
		&:=
		D_A\mathcal R_0(t,A_0(t))[\dot A_0(t)]
		+
		\left.
		\partial_p\mathcal R_p(t,A_0(t))
		\right|_{p=0},
		\label{eq:Rsharp-dot-zero-definition}\\
		\dot b_0(t)
		&:=
		\ee^{\ii\Theta_0(t)}
		\left(
		\dot A_0(t)
		+
		\ii\dot\Theta_0(t)A_0(t)
		\right).
		\label{eq:b-dot-zero-definition}
	\end{align}
	Then
	\begin{equation}\label{eq:Zdot-formula}
		\dot Z_0
		=
		S(1)\dot a_0(1)
		-\ii\kappa
		\int_1^\infty
		t^{-1}
		\left[
		\frac12
		\ee^{\ii\Theta_0}\dot{\mathcal R}_0^\sharp
		+
		\frac{\ii}{2}
		\dot\Theta_0
		\ee^{\ii\Theta_0}\mathcal R_0^\sharp
		-
		\frac14
		\log(2t)
		\ee^{\ii\Theta_0}\mathcal R_0^\sharp
		\right]
		\,\dd t.
	\end{equation}
\end{proposition}

\begin{proof}
	Set
	\[
	g_\delta
	:=
	\frac{b_\delta-b_0}{\delta}.
	\]
	We first obtain an integrable bound for $\partial_tg_\delta$, uniformly in $\delta$. The equation \eqref{eq:exact-b-equation} gives
	\begin{equation}\label{eq:gdelta-exact-equation}
		\begin{aligned}
			\partial_tg_\delta
			=-\ii\kappa t^{-1}\Biggl[&
			2^{-1-\delta/2}t^{-\delta/2}\ee^{\ii\Theta_\delta}
			\frac{\mathcal R_\delta^\sharp-\mathcal R_0^\sharp}{\delta}\\
			&+2^{-1-\delta/2}t^{-\delta/2}
			\frac{\ee^{\ii\Theta_\delta}-\ee^{\ii\Theta_0}}{\delta}
			\mathcal R_0^\sharp
			+\frac{2^{-1-\delta/2}t^{-\delta/2}-\frac12}{\delta}
			\ee^{\ii\Theta_0}\mathcal R_0^\sharp\Biggr].
		\end{aligned}
	\end{equation}
	By the proof of \eqref{eq:conjugation-remainder-L2} and \eqref{eq:a-H1-coupled-improvement},
	\[
	\norm{\mathcal R_0^\sharp(t)}_{L^2}
	\leq
	C\varepsilon^3
	t^{-1/2+C\varepsilon^2}.
	\]
	Furthermore,
	\[
	\abs{
		\frac{
			\ee^{\ii\Theta_\delta}
			-
			\ee^{\ii\Theta_0}
		}{\delta}
	}
	\leq
	\abs{\frac{\Theta_\delta-\Theta_0}{\delta}}.
	\]
	Use the actual polynomial bounds proved in \Cref{lem:conjugation-remainder-quotient,lem:phase-quotient-H1}, $H^1(\R)\hookrightarrow L^\infty(\R)$, and \eqref{eq:time-coefficient-quotient-bound}.
	The three terms inside the brackets in \eqref{eq:gdelta-exact-equation} have $L^2$ bounds, respectively,
\[
Ct^{-1/2+C\varepsilon^2}(1+\log t)^3,\qquad
Ct^{-1/2+C\varepsilon^2}(1+\log t)^4,\qquad
Ct^{-1/2+C\varepsilon^2}(1+\log t).
\]
	For the middle term, the exponent constant is the sum of those for the phase quotient and the remainder. Multiplying by $t^{-1}$ and enlarging the absolute exponent constant gives
	\begin{equation}\label{eq:gdelta-time-integrable}
		\norm{\partial_tg_\delta(t)}_{L^2}
		\leq
		Ct^{-3/2+C\varepsilon^2}
		(1+\log t)^5.
	\end{equation}
	All polynomial exponent constants obtained in these proofs are absolute and independent of $C_0$ and $\varepsilon_0$. There are only finitely many, including the sums in the products above. Choose $C_0$ once, larger than all of them, and then decrease $\varepsilon_0$ to satisfy \eqref{eq:section-four-smallness} and the earlier smallness conditions. Every preceding absorption into $C_0$ is thereby justified. In particular, $C\varepsilon^2\leq C_0\varepsilon_0^2\leq1/32<1/4$, so
\[
\norm{\partial_tg_\delta(t)}_{L^2}
	\leq
	Ct^{-5/4}(1+\log t)^5,
	\]
	which is integrable on $[1,\infty)$.
	
	For every fixed $t$, \Cref{lem:H1-parameter-quotient,lem:conjugation-remainder-quotient,lem:phase-quotient-H1} show that the integrand in \eqref{eq:gdelta-exact-equation} converges in $L^2$ to the time derivative of \eqref{eq:b-dot-zero-definition}. Since
	\[
	g_\delta(1)
	=
	S(1)h_\delta(1)
	\longrightarrow
	S(1)\dot a_0(1)
	\qquad
	\text{in }L^2,
	\]
	Bochner dominated convergence gives
	\[
	\frac{Z_\delta-Z_0}{\delta}
	=
	g_\delta(1)
	+
	\int_1^\infty
	\partial_tg_\delta(t)
	\,\dd t
	\longrightarrow
	\dot Z_0
	\]
	with $\dot Z_0$ given by \eqref{eq:Zdot-formula}.
\end{proof}

\subsection{Expansion of \texorpdfstring{$\Gamma_\delta$}{Gamma}}
\label{subsec:phase-defect-quotient}

Set
\begin{equation}\label{eq:ydelta-definition}
	y_\delta(t)
	:=
	b_\delta(t)-Z_\delta.
\end{equation}
By \eqref{eq:b-Z-tail-Hs},
\begin{equation}\label{eq:ydelta-Hs-tail}
	\norm{y_\delta(t)}_{H^{3/4}}
	+
	\norm{y_\delta(t)}_{L^\infty}
	\leq
	Ct^{-1/16}.
\end{equation}
Since
\[
\frac{y_\delta(t)-y_0(t)}{\delta}
=
g_\delta(t)-\frac{Z_\delta-Z_0}{\delta}
=
-\int_t^\infty \partial_s g_\delta(s)\,\dd s,
\]
\eqref{eq:gdelta-time-integrable} and \eqref{eq:section-four-smallness} give
\begin{equation}\label{eq:ydelta-quotient-tail}
	\norm{
		\frac{y_\delta(t)-y_0(t)}{\delta}
	}_{L^2}
	\leq
	Ct^{-1/4}(1+\log t)^5.
\end{equation}

\begin{lemma}
	\label{lem:phase-defect-density-quotient}
	Define
	\begin{equation}\label{eq:density-difference-definition}
		H_\delta(t)
		:=
		\abs{b_\delta(t)}^{2+\delta}
		-
		\abs{Z_\delta}^{2+\delta}.
	\end{equation}
	Then
	\begin{equation}\label{eq:Hdelta-quotient-bound}
		\norm{
			\frac{H_\delta(t)-H_0(t)}{\delta}
		}_{L^2}
		\leq
		Ct^{-1/16}(1+\log t)^5.
	\end{equation}
	For every fixed $t$, this quotient converges in $L^2$ to
	\begin{equation}\label{eq:Hdot-zero-formula}
		\dot H_0(t)
		:=
		2\operatorname{Re}
		\left(
		\overline{b_0(t)}\dot b_0(t)
		-
		\overline{Z_0}\dot Z_0
		\right)
		+
		\abs{b_0(t)}^2\log\abs{b_0(t)}
		-
		\abs{Z_0}^2\log\abs{Z_0}.
	\end{equation}
\end{lemma}

\begin{proof}
	To retain the decay of $b_\delta-Z_\delta$, subtract the two densities before taking the exponent quotient. Write the changes in the tail, the limiting profile, and the exponent separately. The resulting decomposition is
	\begin{equation}\label{eq:Hdelta-three-part-decomposition}
		\begin{aligned}
			\frac{H_\delta-H_0}{\delta}
			=&\frac{\abs{Z_\delta+y_\delta}^{2+\delta}
				-\abs{Z_\delta+y_0}^{2+\delta}}{\delta}\\
			&+\frac{
				\bigl(\abs{Z_\delta+y_0}^{2+\delta}-\abs{Z_\delta}^{2+\delta}\bigr)
				-\bigl(\abs{Z_0+y_0}^{2+\delta}-\abs{Z_0}^{2+\delta}\bigr)}{\delta}\\
			&+\frac{
				\bigl(\abs{Z_0+y_0}^{2+\delta}-\abs{Z_0}^{2+\delta}\bigr)
				-\bigl(\abs{Z_0+y_0}^{2}-\abs{Z_0}^{2}\bigr)}{\delta}.
		\end{aligned}
	\end{equation}
	
	For the first term, the mean-value formula in $y$ and \eqref{eq:ydelta-quotient-tail} give
	\begin{equation}\label{eq:Hdelta-y-variation}
		\norm{\frac{\abs{Z_\delta+y_\delta}^{2+\delta}
				-\abs{Z_\delta+y_0}^{2+\delta}}{\delta}}_{L^2}
		\leq Ct^{-1/4}(1+\log t)^5.
	\end{equation}
	
	For the second term, retain the common increment $y_0$. Differentiating the increment with respect to its base point gives
	\[
	\begin{aligned}
		&\bigl(\abs{Z_\delta+y_0}^{2+\delta}-\abs{Z_\delta}^{2+\delta}\bigr)
		-\bigl(\abs{Z_0+y_0}^{2+\delta}-\abs{Z_0}^{2+\delta}\bigr)\\
		&\quad=\int_0^1
		\left.D_z\bigl(\abs{z+y_0}^{2+\delta}-\abs{z}^{2+\delta}\bigr)
		\right|_{z=Z_0+\theta(Z_\delta-Z_0)}
		[Z_\delta-Z_0]\,\dd\theta.
	\end{aligned}
	\]
	The derivative inside this integral is the difference of the two density derivatives at points separated by $y_0$. Since $\abs{D_z^2(\abs{z}^{2+p})}\leq C$ on the bounded range of the profiles, every occurrence of $Z_\delta-Z_0$ is accompanied by $y_0$. Hence
	\begin{equation}\label{eq:Hdelta-Z-variation}
	\begin{aligned}
			&\norm{\frac{
					\bigl(\abs{Z_\delta+y_0}^{2+\delta}-\abs{Z_\delta}^{2+\delta}\bigr)
					-\bigl(\abs{Z_0+y_0}^{2+\delta}-\abs{Z_0}^{2+\delta}\bigr)}{\delta}}_{L^2}\\
			&\quad\leq C\norm{y_0(t)}_{L^\infty}
			\norm{\frac{Z_\delta-Z_0}{\delta}}_{L^2}
			\leq Ct^{-1/16},
	\end{aligned}
	\end{equation}
	where \Cref{prop:Z-first-order-expansion} gives the uniform $L^2$ bound for $\frac{Z_\delta-Z_0}{\delta}$.
	
	For the exponent variation, differentiate both densities at the same exponent:
	\[
	\begin{aligned}
		&\frac{
			\bigl(\abs{Z_0+y_0}^{2+\delta}-\abs{Z_0}^{2+\delta}\bigr)
			-\bigl(\abs{Z_0+y_0}^{2}-\abs{Z_0}^{2}\bigr)}{\delta}\\
		&\quad=\int_0^1\Bigl[
		\abs{Z_0+y_0}^{2+\theta\delta}\log\abs{Z_0+y_0}
		-\abs{Z_0}^{2+\theta\delta}\log\abs{Z_0}
		\Bigr]\,\dd\theta.
	\end{aligned}
	\]
	Because
	\[
	\abs{D_z(\abs{z}^{2+p}\log\abs{z})}
	\leq
	C\abs{z}^{1+p}
	\left(
	1+\abs{\log\abs z}
	\right)
	\]
	is uniformly bounded on the range of the profiles and extends continuously at zero,
	\begin{equation}\label{eq:Hdelta-exponent-variation}
		\norm{\frac{
				\bigl(\abs{Z_0+y_0}^{2+\delta}-\abs{Z_0}^{2+\delta}\bigr)
				-\bigl(\abs{Z_0+y_0}^{2}-\abs{Z_0}^{2}\bigr)}{\delta}}_{L^2}
		\leq C\norm{y_0(t)}_{L^2}\leq Ct^{-1/16}.
	\end{equation}
	Combining \eqref{eq:Hdelta-y-variation}, \eqref{eq:Hdelta-Z-variation}, and \eqref{eq:Hdelta-exponent-variation} proves \eqref{eq:Hdelta-quotient-bound}.
	
	For fixed $t$, the state quotients converge in $L^2$ by \Cref{prop:Z-first-order-expansion} and the fixed-time convergence of $g_\delta$. The profiles are uniformly bounded in $H^{3/4}$, so interpolation through $H^{5/8}$ gives the $L^\infty$ convergence needed in the mean-value formulas. Applying the same three formulas in \eqref{eq:Hdelta-three-part-decomposition} therefore gives \eqref{eq:Hdot-zero-formula}.
\end{proof}

\begin{proposition}
	\label{prop:Gamma-first-order-expansion}
	There exists $\dot\Gamma_0\in L^2(\R)$ such that
	\begin{equation}\label{eq:Gamma-first-order-section-five}
		\frac{\Gamma_\delta-\Gamma_0}{\delta}
		\longrightarrow
		\dot\Gamma_0
		\qquad\text{in }L^2.
	\end{equation}
	It is given by
	\begin{equation}\label{eq:Gammadot-formula}
		\dot\Gamma_0
		=
		\frac{\kappa}{2}
		\int_1^\infty
		t^{-1}
		\left[
		\dot H_0(t)
		-
		\frac12\log(2t)H_0(t)
		\right]
		\,\dd t.
	\end{equation}
\end{proposition}

\begin{proof}
	By \eqref{eq:Gamma-delta-definition},
	\begin{equation}\label{eq:Gamma-quotient-decomposition}
		\frac{\Gamma_\delta-\Gamma_0}{\delta}
		=
		\kappa\int_1^\infty
		2^{-1-\delta/2}t^{-1-\delta/2}
		\frac{H_\delta(t)-H_0(t)}{\delta}
		\,\dd t
		+
		\kappa\int_1^\infty
		t^{-1}
		\frac{2^{-1-\delta/2}t^{-\delta/2}-2^{-1}}{\delta}
		H_0(t)
		\,\dd t.
	\end{equation}
	By \eqref{eq:Hdelta-quotient-bound}, the first integrand is
	bounded in $L^2$ by
	\[
	Ct^{-1-1/16}(1+\log t)^5.
	\]
	Moreover, \eqref{eq:ydelta-Hs-tail} gives
	\[
	\norm{H_0(t)}_{L^2}
	\leq
	Ct^{-1/16}.
	\]
	Together with \eqref{eq:time-coefficient-quotient-bound}, the second integrand
	is bounded in $L^2$ by
	\[
	Ct^{-1-1/16}(1+\log t).
	\]
	Both bounds are integrable on $[1,\infty)$.
	
	Finally,
	\[
	\left.
	\partial_\delta\bigl(2^{-1-\delta/2}t^{-\delta/2}\bigr)
	\right|_{\delta=0}
	=
	-\frac14\log(2t).
	\]
	The fixed-time convergence from \Cref{lem:phase-defect-density-quotient} and Bochner dominated convergence in \eqref{eq:Gamma-quotient-decomposition} prove \eqref{eq:Gamma-first-order-section-five} and \eqref{eq:Gammadot-formula}.
\end{proof}

\subsection{Expansion of \texorpdfstring{$W_\delta$}{W} and of the nonlinear coefficient}
\label{subsec:profile-coefficient-expansion}

\begin{proof}[Proof of \Cref{prop:first-order-profile-expansion}]
	\Cref{prop:Z-first-order-expansion,prop:Gamma-first-order-expansion}
	give the first-order expansions of $Z_\delta$ and $\Gamma_\delta$. In particular,
	\begin{equation}\label{eq:profile-L2-Odelta}
		\norm{Z_\delta-Z_0}_{L^2}
		+
		\norm{\Gamma_\delta-\Gamma_0}_{L^2}
		\leq
		C\delta.
	\end{equation}
	To pass to the limit in the products defining $W_\delta$ and $2^{-1-\delta/2}\abs{W_\delta}^{2+\delta}$, we also need convergence in $L^\infty$. The differences are uniformly bounded in $H^{3/4}$, so interpolation gives
	\[
	\norm{f}_{H^{5/8}}
	\leq
	C\norm{f}_{L^2}^{1/6}
	\norm{f}_{H^{3/4}}^{5/6}.
	\]
	Since $H^{5/8}(\R)\hookrightarrow L^\infty(\R)$,
	\begin{equation}\label{eq:profile-Linfty-differences}
		\norm{Z_\delta-Z_0}_{L^\infty}
		+
		\norm{\Gamma_\delta-\Gamma_0}_{L^\infty}
		\leq
		C\delta^{1/6}.
	\end{equation}
	
	Taylor's formula and \eqref{eq:profile-Linfty-differences} give
	\[
		\norm{\left[
			\frac{\ee^{-\ii\Gamma_\delta}-\ee^{-\ii\Gamma_0}}{\delta}
			+\ii\ee^{-\ii\Gamma_0}\frac{\Gamma_\delta-\Gamma_0}{\delta}
			\right]Z_0}_{L^2}\leq C\norm{\Gamma_\delta-\Gamma_0}_{L^\infty}
		\norm{\frac{\Gamma_\delta-\Gamma_0}{\delta}}_{L^2}
		\norm{Z_0}_{L^\infty}\longrightarrow0.
	\]
	
	Since $W_\delta=\ee^{-\ii\Gamma_\delta}Z_\delta$,
	\[
	\frac{W_\delta-W_0}{\delta}
	=
	\ee^{-\ii\Gamma_\delta}\frac{Z_\delta-Z_0}{\delta}
	+
	\frac{
		\ee^{-\ii\Gamma_\delta}
		-
		\ee^{-\ii\Gamma_0}
	}{\delta}
	Z_0.
	\]
	By \eqref{eq:profile-Linfty-differences} and the $L^2$ convergences of $\frac{Z_\delta-Z_0}{\delta}$ and $\frac{\Gamma_\delta-\Gamma_0}{\delta}$, the two terms
	converge in $L^2$. Therefore
	\begin{equation}\label{eq:Wdot-formula}
		\frac{W_\delta-W_0}{\delta}
		\longrightarrow
		\dot W_0
		:=
		\ee^{-\ii\Gamma_0}
		\left(
		\dot Z_0-\ii\dot\Gamma_0Z_0
		\right)
		\qquad
		\text{in }L^2.
	\end{equation}
	
	It remains to expand $2^{-1-\delta/2}\abs{Z_\delta}^{2+\delta}$, which equals the coefficient $2^{-1-\delta/2}\abs{W_\delta}^{2+\delta}$.
	First,
	\begin{equation}\label{eq:Z-square-quotient}
		\frac{
			\abs{Z_\delta}^2-\abs{Z_0}^2
		}{\delta}
		=
		2\operatorname{Re}(\overline{Z_0}\frac{Z_\delta-Z_0}{\delta})
		+
		\frac{\abs{Z_\delta-Z_0}^2}{\delta}.
	\end{equation}
	The first term converges in $L^2$, while
	\[
	\norm{\frac{\abs{Z_\delta-Z_0}^2}{\delta}}_{L^2}
	\leq
	\norm{Z_\delta-Z_0}_{L^\infty}
	\norm{\frac{Z_\delta-Z_0}{\delta}}_{L^2}
	\longrightarrow0
	\]
	by \eqref{eq:profile-Linfty-differences}.
	
	We next consider the moving exponent. For $0<r\leq1$,
	\[
	\frac{\dd}{\dd r}
	\left(
	\frac{r^{2+\delta}-r^2}{\delta}
	\right)
	=
	r
	\left[
	2\frac{r^\delta-1}{\delta}
	+
	r^\delta
	\right].
	\]
	Since
	\[
	\abs{
		\frac{r^\delta-1}{\delta}
	}
	\leq
	\abs{\log r},
	\]
	the derivative is bounded uniformly in $\delta$ by
	\[
	Cr(1+\abs{\log r})
	\leq C,
	\]
	with continuous extension at $r=0$. Hence
	\[
	\norm{
		\frac{
			\abs{Z_\delta}^{2+\delta}
			-
			\abs{Z_\delta}^2
		}{\delta}
		-
		\frac{
			\abs{Z_0}^{2+\delta}
			-
			\abs{Z_0}^2
		}{\delta}
	}_{L^2}
	\leq
	C\norm{Z_\delta-Z_0}_{L^2}
	\longrightarrow0.
	\]
	For the fixed argument,
	\[
	\frac{
		\abs{Z_0}^{2+\delta}
		-
		\abs{Z_0}^2
	}{\delta}
	\longrightarrow
	\abs{Z_0}^2\log\abs{Z_0}
	\qquad
	\text{in }L^2.
	\]
	Indeed, for $0\leq r\leq1$,
	\[
	\abs{
		\frac{r^{2+\delta}-r^2}{\delta}
	}
	\leq
	r^2\abs{\log r}
	\lesssim
	r^{3/2},
	\]
	and $Z_0\in L^2\cap L^\infty$.
	
	Finally,
	\[
	2^{-1-\delta/2}
	=
	\frac12
	-
	\frac{\delta}{4}\log2
	+
	O(\delta^2).
	\]
	Combining these expansions proves
	\eqref{eq:coefficient-first-order-main}.
\end{proof}

\section{Transition-phase analysis}
\label{sec:transition-phase-analysis}

We combine the uniform asymptotics with the first-order coefficient expansion from \Cref{prop:first-order-profile-expansion}. No new evolution estimate is needed. Throughout the section, let
\[
\delta_n\to0^+,\qquad t_n\to\infty,
\]
and set
\begin{equation}\label{eq:transition-lambda-definition}
	\lambda_n:=\delta_n\log t_n.
\end{equation}
The clock identity is
\begin{equation}\label{eq:transition-exact-clock-function}
	P_{\delta_n}(t_n)=\frac{2(1-\ee^{-\lambda_n/2})}{\delta_n}.
\end{equation}
We retain this identity until a stronger scale condition justifies an absolute replacement of the clock.

\subsection{Exact-clock reduction}
\label{subsec:exact-clock-reduction}

\begin{proposition}
	\label{prop:exact-clock-reduction}
	Under the assumptions of \Cref{prop:uniform-asymptotics},
	\begin{equation}\label{eq:phase-reduction}
		\begin{aligned}
			&P_{\delta_n}(t_n)2^{-1-\delta_n/2}\abs{W_{\delta_n}}^{2+\delta_n}\\
			&\quad=\frac12P_{\delta_n}(t_n)\abs{W_0}^2+(1-\ee^{-\lambda_n/2})\Bigl[
			2\operatorname{Re}(\overline{Z_0}\dot Z_0)
			+\abs{Z_0}^2\left(\log\abs{Z_0}-\frac12\log2\right)\Bigr]
			+o_{L_v^2}(1).
		\end{aligned}
	\end{equation}
	Consequently,
	\begin{equation}\label{eq:amplitude-reduction}
		\begin{aligned}
			a_{\delta_n}(t_n)
			=&\exp\left\{-\ii\kappa \frac12P_{\delta_n}(t_n)\abs{W_0}^2\right\}W_0\\
			&\times\exp\Biggl\{-\ii\kappa (1-\ee^{-\lambda_n/2})\Bigl[
			2\operatorname{Re}(\overline{Z_0}\dot Z_0)+\abs{Z_0}^2\left(\log\abs{Z_0}-\frac12\log2\right)
			\Bigr]\Biggr\}+o_{L_v^2}(1).
		\end{aligned}
	\end{equation}
\end{proposition}

\begin{proof}
	By \eqref{eq:coefficient-first-order-main}, the error in the first-order coefficient expansion is $o_{L^2}(\delta_n)$. Since
	\[
	0\leq P_{\delta_n}(t_n)\leq\frac2{\delta_n},
	\qquad
	\frac{\delta_n}{2}P_{\delta_n}(t_n)=1-\ee^{-\lambda_n/2},
	\]
	multiplying that expansion by $P_{\delta_n}(t_n)$ gives \eqref{eq:phase-reduction}, with an $o_{L^2}(1)$ error. The phase in \eqref{eq:uniform-asymptotics} is
	\[
	\Phi_{\delta_n}[W_{\delta_n}](t_n)
	=\kappa P_{\delta_n}(t_n)
	2^{-1-\delta_n/2}\abs{W_{\delta_n}}^{2+\delta_n}.
	\]
	First replace $W_{\delta_n}$ outside the exponential by $W_0$, at a cost of $\norm{W_{\delta_n}-W_0}_{L^2}\to0$ by \eqref{eq:W-first-order-main}. Then the additional error is at most $\norm{W_0}_{L^\infty}$ times the $L^2$ error in \eqref{eq:phase-reduction}. Together with the uniform error $C\varepsilon t_n^{-1/16}$, these bounds prove \eqref{eq:amplitude-reduction}. Here $W_0\in H^{3/4}(\R)\hookrightarrow L^\infty(\R)$.
\end{proof}

\subsection{The three regimes}
\label{subsec:three-exact-clock-regimes}

\begin{theorem}
	\label{thm:exact-clock-trichotomy}
	Under the assumptions of \Cref{prop:exact-clock-reduction}, the following assertions hold.
	
	If $\lambda_n\to0$, then
	\begin{equation}\label{eq:exact-clock-regime-one-amplitude}
		a_{\delta_n}(t_n)
		=\ee^{-\frac{\ii\kappa}{2}P_{\delta_n}(t_n)\abs{W_0}^2}W_0
		+o_{L_v^2}(1),
	\end{equation}
	and
	\begin{equation}\label{eq:exact-clock-regime-one-normalized}
		\frac{P_{\delta_n}(t_n)2^{-1-\delta_n/2}
			\abs{W_{\delta_n}}^{2+\delta_n}}{\log t_n}
		\longrightarrow\frac12\abs{W_0}^2
		\qquad\text{in }L_v^2.
	\end{equation}
	
	If $\lambda_n\to\Lambda\in(0,\infty)$, then
	\begin{equation}\label{eq:exact-clock-regime-two-amplitude}
		\begin{aligned}
			a_{\delta_n}(t_n)
			=&\exp\left\{-\ii\kappa \frac{1-\ee^{-\lambda_n/2}}{\delta_n}\abs{W_0}^2\right\}W_0\\
			&\times\exp\Biggl\{-\ii\kappa (1-\ee^{-\Lambda/2})\Bigl[
			2\operatorname{Re}(\overline{Z_0}\dot Z_0)+\abs{Z_0}^2\left(\log\abs{Z_0}-\frac12\log2\right)
			\Bigr]\Biggr\}+o_{L_v^2}(1).
		\end{aligned}
	\end{equation}
	and
	\begin{equation}\label{eq:exact-clock-regime-two-normalized}
		\delta_nP_{\delta_n}(t_n)2^{-1-\delta_n/2}
		\abs{W_{\delta_n}}^{2+\delta_n}
		\longrightarrow(1-\ee^{-\Lambda/2})\abs{W_0}^2
		\qquad\text{in }L_v^2.
	\end{equation}
	The second exponential in \eqref{eq:exact-clock-regime-two-amplitude} contains a finite phase correction and cannot be absorbed into the error.
	
	If $\lambda_n\to\infty$, then
	\begin{equation}\label{eq:exact-clock-regime-three-amplitude}
		\begin{aligned}
			a_{\delta_n}(t_n)
			=&\exp\left\{-\ii\kappa \frac{1-\ee^{-\lambda_n/2}}{\delta_n}\abs{W_0}^2\right\}W_0\\
			&\times\exp\Biggl\{-\ii\kappa \Bigl[
			2\operatorname{Re}(\overline{Z_0}\dot Z_0)+\abs{Z_0}^2\left(\log\abs{Z_0}-\frac12\log2\right)
			\Bigr]\Biggr\}+o_{L_v^2}(1).
		\end{aligned}
	\end{equation}
	and
	\begin{equation}\label{eq:exact-clock-regime-three-normalized}
		\delta_nP_{\delta_n}(t_n)2^{-1-\delta_n/2}
		\abs{W_{\delta_n}}^{2+\delta_n}
		\longrightarrow\abs{W_0}^2
		\qquad\text{in }L_v^2.
	\end{equation}
\end{theorem}

\begin{proof}
	If $\lambda_n\to0$, then $1-\ee^{-\lambda_n/2}\to0$. The explicit first-order coefficient in \eqref{eq:coefficient-first-order-main} belongs to $L^2$. Thus $W_0\in L^\infty$ allow the second exponential in \eqref{eq:amplitude-reduction} to be replaced by one, giving \eqref{eq:exact-clock-regime-one-amplitude}. Moreover,
	\[
	\frac{P_{\delta_n}(t_n)}{\log t_n}
	=\frac{2(1-\ee^{-\lambda_n/2})}{\lambda_n}\longrightarrow1.
	\]
	Together with
	\[
	2^{-1-\delta_n/2}\abs{W_{\delta_n}}^{2+\delta_n}
	\longrightarrow\frac12\abs{W_0}^2\quad\text{in }L^2,
	\]
	this proves \eqref{eq:exact-clock-regime-one-normalized}.
	
	If $\lambda_n\to\Lambda\in(0,\infty)$, then $1-\ee^{-\lambda_n/2}\to1-\ee^{-\Lambda/2}$. Apply the same exponential estimate only to the finite correction in \eqref{eq:amplitude-reduction}; the leading clock remains unchanged. This gives \eqref{eq:exact-clock-regime-two-amplitude}.
	Also,
	\[
		\delta_nP_{\delta_n}(t_n)2^{-1-\delta_n/2}
		\abs{W_{\delta_n}}^{2+\delta_n}
		=2(1-\ee^{-\lambda_n/2})2^{-1-\delta_n/2}
		\abs{W_{\delta_n}}^{2+\delta_n}
		\longrightarrow(1-\ee^{-\Lambda/2})\abs{W_0}^2
		\quad\text{in }L^2.
	\]
	Finally, if $\lambda_n\to\infty$, the same argument uses
	$1-\ee^{-\lambda_n/2}\to1$ and yields
	\eqref{eq:exact-clock-regime-three-amplitude} and
	\eqref{eq:exact-clock-regime-three-normalized}.
\end{proof}

\subsection{Taylor hierarchy}
\label{subsec:cubic-side-hierarchy}

The relative equivalence $P_\delta(t)\sim\log t$ does not by itself give the absolute phase accuracy needed here. The clock admits the following Taylor estimate.

\begin{lemma}[Taylor remainder for the clock]
	\label{lem:exact-clock-taylor-remainder}
	For every integer $m\geq0$, $\delta>0$, and $t\geq1$,
	\begin{equation}\label{eq:exact-clock-taylor-remainder}
		\abs{P_\delta(t)-P_\delta^{(m)}(t)}
		\leq\frac{\delta^{m+1}(\log t)^{m+2}}{2^{m+1}(m+2)!}.
	\end{equation}
\end{lemma}

\begin{proof}
	Taylor's theorem, with $\delta\log t/2\geq0$, gives directly
	\[
	\abs{
		1-\ee^{-\delta\log t/2}
		-\sum_{j=1}^{m+1}\frac{(-1)^{j+1}}{j!}
		\left(\frac{\delta\log t}{2}\right)^j}
	\leq\frac{(\delta\log t/2)^{m+2}}{(m+2)!}.
	\]
	Multiply by $2/\delta$ and set $k=j-1$ in the sum to obtain \eqref{eq:exact-clock-taylor-remainder}.
\end{proof}

\begin{corollary}[Truncated-clock asymptotics]
	\label{cor:cubic-hierarchy-asymptotics}
	Under the assumptions of \Cref{thm:exact-clock-trichotomy}, fix an integer $m\geq0$. If
	\begin{equation}\label{eq:cubic-hierarchy-assumptions}
		\delta_n^{m+1}(\log t_n)^{m+2}\longrightarrow0,
	\end{equation}
	then
	\begin{equation}\label{eq:cubic-hierarchy-clock-convergence}
		P_{\delta_n}(t_n)-P_{\delta_n}^{(m)}(t_n)\longrightarrow0
	\end{equation}
	and
	\begin{equation}\label{eq:cubic-hierarchy-amplitude}
		a_{\delta_n}(t_n)
		=\ee^{-\frac{\ii\kappa}{2}P_{\delta_n}^{(m)}(t_n)\abs{W_0}^2}W_0
		+o_{L_v^2}(1).
	\end{equation}
	For $m=0$, this gives $P_{\delta_n}(t_n)=\log t_n+o(1)$ under
	$\delta_n(\log t_n)^2\to0$.
	
	For $m\geq1$, at the scale
	\begin{equation}\label{eq:cubic-scale-hierarchy-scale}
		\log t_n\sim c\,\delta_n^{-m/(m+1)},\qquad c>0,
	\end{equation}
	condition \eqref{eq:cubic-hierarchy-assumptions} holds and
	\begin{equation}\label{eq:cubic-scale-hierarchy-term}
		P_{\delta_n}^{(m)}(t_n)-P_{\delta_n}^{(m-1)}(t_n)
		\longrightarrow\frac{(-1)^m c^{m+1}}{2^m(m+1)!}.
	\end{equation}
	Thus the newly visible finite phase contribution is
	\begin{equation}\label{eq:cubic-scale-hierarchy-phase}
		\frac{(-1)^m c^{m+1}}{2^{m+1}(m+1)!}\abs{W_0}^2.
	\end{equation}
	The first examples are
	\[
	-\frac{c^2}{8}\abs{W_0}^2,\qquad
	\frac{c^3}{48}\abs{W_0}^2,\qquad
	-\frac{c^4}{384}\abs{W_0}^2,
	\]
	at
	\[
	\log t\sim c\,\delta^{-1/2},\qquad
	\log t\sim c\,\delta^{-2/3},\qquad
	\log t\sim c\,\delta^{-3/4},
	\]
	respectively.
\end{corollary}

\begin{proof}
	The scalar convergenc \eqref{eq:cubic-hierarchy-clock-convergence}
	follows from \Cref{lem:exact-clock-taylor-remainder}. Moreover,
	\[
	(\delta_n\log t_n)^{m+1}
	=\frac{\delta_n^{m+1}(\log t_n)^{m+2}}{\log t_n}\longrightarrow0,
	\]
	so $\lambda_n\to0$ and \eqref{eq:exact-clock-regime-one-amplitude}
	applies. Since $W_0\in L^\infty$ and $\abs{W_0}^2\in L^2$,
	\[
	\begin{aligned}
		&\snorm{\left(
			\ee^{-\frac{\ii\kappa}{2}P_{\delta_n}(t_n)\abs{W_0}^2}
			-\ee^{-\frac{\ii\kappa}{2}P_{\delta_n}^{(m)}(t_n)\abs{W_0}^2}
			\right)W_0}_{L^2}\\
		&\quad\leq\frac12\norm{W_0}_{L^\infty}\norm{\abs{W_0}^2}_{L^2}
		\abs{P_{\delta_n}(t_n)-P_{\delta_n}^{(m)}(t_n)}\longrightarrow0.
	\end{aligned}
	\]
	This proves \eqref{eq:cubic-hierarchy-amplitude}. Under \eqref{eq:cubic-scale-hierarchy-scale},
	\[
	\delta_n^{m+1}(\log t_n)^{m+2}
	\sim c^{m+2}\delta_n^{1/(m+1)}\longrightarrow0,
	\qquad
	\delta_n^m(\log t_n)^{m+1}\longrightarrow c^{m+1}.
	\]
	This proves \eqref{eq:cubic-scale-hierarchy-term}.
\end{proof}

\subsection{Refinements in the transition and saturated regimes}
\label{subsec:transition-saturation-refinements}

The Taylor hierarchy in \Cref{cor:cubic-hierarchy-asymptotics} gives finer phase approximations when $\delta\log t\to0$. In the transition and saturated regimes, replacing the exact clock by its limiting form likewise requires absolute accuracy. The sufficient conditions are stated explicitly below.

\begin{corollary}
	\label{cor:simplified-clock-asymptotics}
	Under the assumptions of \Cref{thm:exact-clock-trichotomy}, let $\Lambda\in(0,\infty)$. If
	\begin{equation}\label{eq:regime-two-stronger-clock-condition}
		\lambda_n-\Lambda=o(\delta_n),
	\end{equation}
	then
	\begin{equation}\label{eq:simplified-clock-regime-two}
		\begin{aligned}
			a_{\delta_n}(t_n)
			=&\exp\left\{-\ii\kappa \frac{1-\ee^{-\Lambda/2}}{\delta_n}\abs{W_0}^2\right\}W_0\\
			&\times\exp\Biggl\{-\ii\kappa (1-\ee^{-\Lambda/2})\Bigl[
			2\operatorname{Re}(\overline{Z_0}\dot Z_0)+\abs{Z_0}^2\left(\log\abs{Z_0}-\frac12\log2\right)
			\Bigr]\Biggr\}+o_{L_v^2}(1).
		\end{aligned}
	\end{equation}
	
	If
	\begin{equation}\label{eq:regime-three-stronger-clock-condition}
		\frac{\ee^{-\lambda_n/2}}{\delta_n}\longrightarrow0,
	\end{equation}
	then
	\begin{equation}\label{eq:simplified-clock-regime-three}
			a_{\delta_n}(t_n)
			=\exp\left\{-\ii\kappa \frac{\abs{W_0}^2}{\delta_n}\right\}W_0\exp\Biggl\{-\ii\kappa \Bigl[
			2\operatorname{Re}(\overline{Z_0}\dot Z_0)+\abs{Z_0}^2\left(\log\abs{Z_0}-\frac12\log2\right)
			\Bigr]\Biggr\}+o_{L_v^2}(1).
	\end{equation}
\end{corollary}

\begin{proof}
	Condition \eqref{eq:regime-two-stronger-clock-condition} implies
	$\lambda_n\to\Lambda$. Since
	\[
	\frac{\dd}{\dd\lambda}\bigl[2(1-\ee^{-\lambda/2})\bigr]
	=\ee^{-\lambda/2},
	\]
	the mean-value theorem gives
	\[
	\frac{2(1-\ee^{-\lambda_n/2})-2(1-\ee^{-\Lambda/2})}{\delta_n}
	\longrightarrow0.
	\]
	Thus \eqref{eq:exact-clock-regime-two-amplitude} give \eqref{eq:simplified-clock-regime-two}.
	
	Condition \eqref{eq:regime-three-stronger-clock-condition} implies $\lambda_n\to\infty$. Since
	\begin{equation}\label{eq:saturated-clock-defect-section-six}
		P_{\delta_n}(t_n)-\frac2{\delta_n}
		=-\frac2{\delta_n}\ee^{-\lambda_n/2},
	\end{equation}
	\eqref{eq:exact-clock-regime-three-amplitude} gives \eqref{eq:simplified-clock-regime-three}.
\end{proof}

\begin{remark}
	\label{rem:finite-transition-clock-correction}
	Condition \eqref{eq:regime-two-stronger-clock-condition} is generically necessary for replacing the leading exact clock by $2(1-\ee^{-\Lambda/2})/\delta_n$ with absolute $o(1)$ phase error when $W_0$ is nontrivial. Indeed, if
	\[
	\frac{\lambda_n-\Lambda}{\delta_n}\longrightarrow\varsigma\in\R,
	\]
	then
	\begin{equation}\label{eq:finite-transition-clock-correction}
		\frac{(1-\ee^{-\lambda_n/2})-(1-\ee^{-\Lambda/2})}{\delta_n}
		\abs{W_0}^2
		\longrightarrow\frac{\varsigma}{2}\ee^{-\Lambda/2}\abs{W_0}^2
		\quad\text{in }L^2.
	\end{equation}
	Thus the finite correction
	$\frac{\varsigma}{2}\ee^{-\Lambda/2}\abs{W_0}^2$ remains.
\end{remark}

\begin{remark}
	\label{rem:finite-saturation-defect}
	Condition \eqref{eq:regime-three-stronger-clock-condition} is equivalent to
	\begin{equation}\label{eq:saturation-condition-equivalences}
		\frac{t_n^{-\delta_n/2}}{\delta_n}\longrightarrow0
		\qquad\text{and to}\qquad
		\lambda_n-2\log(1/\delta_n)\longrightarrow+\infty.
	\end{equation}
	It is necessary for the scalar clock itself to differ from $2/\delta_n$ by $o(1)$, and generically necessary for the corresponding phase replacement when $W_0$ is nontrivial. If instead
	\[
	\frac{\ee^{-\lambda_n/2}}{\delta_n}\longrightarrow\tau\in(0,\infty),
	\]
	then \eqref{eq:saturated-clock-defect-section-six} leaves the finite correction
	\begin{equation}\label{eq:finite-saturation-clock-correction}
		-\tau\abs{W_0}^2.
	\end{equation}
\end{remark}

For fixed $\delta>0$, the limiting nonlinear phase is $\kappa 2^{-\delta/2}\abs{W_\delta}^{2+\delta}/\delta$. The expansion \eqref{eq:coefficient-first-order-main} gives
\[
	\frac{\kappa}{\delta}2^{-\delta/2}\abs{W_\delta}^{2+\delta}
	=\kappa\Biggl[\frac{\abs{W_0}^2}{\delta}
	+2\operatorname{Re}(\overline{Z_0}\dot Z_0)
	+\abs{Z_0}^2\left(\log\abs{Z_0}-\frac12\log2\right)
	\Biggr]+o_{L_v^2}(1)
\]
as $\delta\downarrow0$. This is a first-order expansion, not an exact identity for fixed $\delta$.

\section{Proofs of the main theorems}
\label{sec:proof-main-theorem}

The common inputs ar \Cref{prop:uniform-asymptotics,prop:first-order-profile-expansion}. They provide $W_0,Z_0\in H^{3/4}$, $\dot Z_0\in L^2$, and the explicit real-valued $L^2$ coefficient in \eqref{eq:coefficient-first-order-main}. We choose $\varepsilon_0$ to satisfy the smallness
conditions in both propositions.

\begin{proof}[Proof of \Cref{thm:main-time-scale-hierarchy}]
	Under \eqref{eq:main-hierarchy-assumptions}, \Cref{cor:cubic-hierarchy-asymptotics} gives \eqref{eq:main-cubic-hierarchy-clock} and the truncated-clock asymptotic in $L_v^2$. The unitary self-similar map \eqref{eq:self-similar-amplitude-definition} then gives \eqref{eq:main-cubic-hierarchy-physical}.
	
	Under \eqref{eq:main-cubic-scale},
	\[
	\log t_n
	\sim
	c\,\delta_n^{-m/(m+1)}.
	\]
	The calculation in \Cref{cor:cubic-hierarchy-asymptotics} shows that \eqref{eq:main-hierarchy-assumptions} holds at this scale and that
	\[
	\delta_n^m(\log t_n)^{m+1}
	\longrightarrow
	c^{m+1}.
	\]
	Multiplying \eqref{eq:cubic-scale-hierarchy-term} by $\abs{W_0}^2/2\in L^2$ gives \eqref{eq:main-cubic-scale-contribution}.
\end{proof}

\begin{proof}[Proof of \Cref{thm:exact_clock}]
	For the same solutions and profiles,
	\Cref{prop:exact-clock-reduction} gives
	\[
		\begin{aligned}
			a_{\delta_n}(t_n)
			=&\exp\left\{-\ii\kappa \frac12P_{\delta_n}(t_n)\abs{W_0}^2\right\}W_0\\
			&\times\exp\Biggl\{-\ii\kappa (1-\ee^{-\lambda_n/2})\Bigl[
			2\operatorname{Re}(\overline{Z_0}\dot Z_0)
			+\abs{Z_0}^2\left(\log\abs{Z_0}-\frac12\log2\right)
			\Bigr]\Biggr\}+o_{L_v^2}(1).
		\end{aligned}
	\]
	The unitary self-similar map \eqref{eq:self-similar-amplitude-definition} gives \eqref{eq:main-exact-clock-physical}.
	
	The three phase replacements follow from \Cref{thm:exact-clock-trichotomy}. Their $o_{L_v^2}(1)$ errors are preserved by the same unitary map, which gives the corresponding $L_x^2$ asymptotics.
\end{proof}

\section*{Acknowledgment}
We thank Chulkwang Kwak for asking a question that motivated this work.
Y. Cho was supported by the research funds of Jeonbuk National University in 2026 and by the National Research Foundation of Korea (NRF) grant funded by the Korea government (MSIT) (RS-2024-00333393). 
J. Lee was partially supported by the Global-Learning \& Academic research institution for Master's-PhD students, and Postdocs (G-LAMP) Program of the National Research Foundation of Korea (NRF) grant funded by the Ministry of Education (No. RS-2025-25442355), the NRF grant funded by the Korea government (MSIT) (No. RS-2026-25589357), and a grant from Kyung Hee University in 2026 (KHU-20262263).

\bibliographystyle{alpha}
\bibliography{reference}

\end{document}